\documentclass[12pt]{amsart}

\usepackage{times}
\usepackage{amsmath}
\usepackage{verbatim}
\usepackage{amsfonts}
\usepackage{amssymb}
\usepackage{color}
\usepackage{bbm}
\usepackage{amsthm}
\usepackage{amscd}
\usepackage[all]{xy}
\usepackage{multirow}
\usepackage{tikz}
\usepackage{tikz-cd}
\usepackage{mathtools}
\usepackage{cite}
\usepackage{caption}
\usepackage{float}

\DeclareMathOperator{\Hom}{Hom}
\DeclareMathOperator{\End}{End}
\DeclareMathOperator{\Ext}{Ext}

\DeclareMathOperator{\grmod}{-grmod}

\DeclareMathOperator{\stab}{-stab}
\DeclareMathOperator{\grstab}{-grstab}
\DeclareMathOperator{\dgstab}{-dgstab}
\DeclareMathOperator{\Z}{\mathbb{Z}}
\let\amsamp=&

\newtheorem{theorem}{Theorem}[section]
\newtheorem{lemma}[theorem]{Lemma}
\newtheorem{proposition}[theorem]{Proposition}
\newtheorem{corollary}[theorem]{Corollary}
\theoremstyle{definition}
\newtheorem{definition}[theorem]{Definition}
\theoremstyle{remark}
\newtheorem*{remark}{Remark}
\newtheorem*{example}{Example}

\begin{document}

\title{Perverse Equivalences and Dg-Stable Combinatorics}
\author{Jeremy R. B. Brightbill}
\date{\today}

\maketitle

\begin{abstract}
Chuang and Rouquier \cite{chuang2017perverse} describe an action by perverse equivalences on the set of bases of a triangulated category of Calabi-Yau dimension $-1$.  We develop an analogue of their theory for Calabi-Yau categories of dimension $w<0$ and show it is equivalent to the mutation theory of $w$-simple-minded systems.

Given a non-positively graded, finite-dimensional symmetric algebra $A$, we show that the differential graded stable category of $A$ has negative Calabi-Yau dimension.  When $A$ is a Brauer tree algebra, we construct a combinatorial model of the dg-stable category and show that perverse equivalences act transitively on the set of $|w|$-bases.
\end{abstract}


\section{Introduction}

Perverse equivalences are equivalences of triangulated categories performed with respect to a stratification.  They were first used by Chuang and Rouquier \cite{chuang2008derived} in their proof of Brou\'e's abelian defect conjecture for symmetric groups; the theory of perverse equivalences was later formalized by the same authors in \cite{chuang2017perverse}.  In this work, Chuang and Rouquier define an action of perverse equivalences on a collection of t-structures, parametrized by tilting complexes, in the bounded derived category of a symmetric algebra.  The further study of this action is the primary motivation of this paper; we shall consider the case of a non-positively graded Brauer tree algebra, $A$.  In doing so, we make two modifications to this action.

The first modification is of the ambient triangulated category.  The bounded derived category of $A$ has many t-structures, and the action of perverse equivalences exhibits braid-like relations, twisted by some other structure.  Viewing the graded algebra $A$ as a dg-algebra with zero differential, we consider the differential graded stable category, $A\dgstab$, whose properties are discussed in \cite{brightbill2018differential}.  The dg-stable category is defined to be the quotient $D^b_{dg}(A)/D^{perf}_{dg}(A)$ of the bounded derived category of (finite-dimensional) dg-modules by the thick subcategory generated by $A$; we can also express $A\dgstab$ as the triangulated hull of the orbit category $A\grstab/\Omega(1)$.  By moving to this setting, we are effectively dividing out the braid group action and studying the residual structure.

Since all projective modules become zero in $A$-dgstab, there are no longer tilting complexes, or even t-structures.  It is therefore necessary to modify the action itself.  Chuang and Rouquier successfully adapt their action to the stable module category and, more generally, to Calabi-Yau categories of dimension $-1$; in this setting, the perverse equivalences act on the set of bases of the category.  The dg-stable category is a generalization of the stable category; the two notions coincide when $A$ is concentrated in degree zero (i.e. ungraded).  One interesting feature of the dg-stable category is the interplay between the grading data of $A$ and the homological structure of the corresponding category of dg-modules; by viewing the grading data of $A$ a parameter, one can study the behavior of the resulting family of triangulated categories.  Up to Morita equivalence, the grading on $A$ is determined by the degree of its socle; when the socle of $A$ is concentrated in degree $-d \le 0$, $A\dgstab$ will be $-(d+1)$-Calabi-Yau.  For any $w<0$, we define an action of perverse equivalences on the set of ``$|w|$-bases'' of a $w$-Calabi-Yau triangulated category.

In the negative Calabi-Yau setting, perverse equivalences are essentially equivalent to the mutation theory of simple-minded systems.  Simple-minded systems were first defined by Koenig and Liu \cite{koenig2011simple} for the stable category of an Artin algebra; Dugas \cite{dugas2015torsion} defines simple-minded systems in an arbitrary Hom-finite Krull-Schmidt category and develops their mutation theory.  Just as the summands in a tilting complex satisfy orthogonality and generating conditions identical to the projective modules in the derived category, simple-minded systems mimic the behavior of simple modules in the stable category.  Coelho Sim\~oes \cite{simoes2017mutations} introduces $|w|$-simple-minded systems in the setting of a $w$-Calabi-Yau category ($w < 0$), and, with Pauksztello, further develops their theory in \cite{simoes2018simple}.  Though we shall primarily use the language of Chuang and Rouquier throughout this paper, we make precise the relationship between the two perspectives.

Once the necessary machinery is in place, we study the action of perverse equivalences on $A\dgstab$ via a combinatorial model, in which objects of $A\dgstab$ are represented by interlocking beads of varying lengths on a circular wire.  $|w|$-bases (or, equivalently, $|w|$-simple-minded systems), correspond to maximal non-overlapping configurations of beads, and perverse equivalences (i.e. mutations) act via physically intuitive transformations of beads.  Our main result establishes transitivity of the action; hence every $|w|$-basis can be obtained via applying successive perverse equivalences to the original collection of simple $A$-modules.

Late in the writing of this paper, the author learned that the category $A\dgstab$ is isomorphic to a category $\mathcal{C}_{|w|}(\mathcal{Q})$ studied by Coelho Sim\~oes \cite{simoes2015hom}, who describes its $|w|$-simple-minded systems via a combinatorial model involving diagonals of an $N$-gon.  In \cite{simoes2017mutations}, Coelho Sim\~oes describes the mutation theory of $|w|$-simple-minded systems in a larger category $T_w$, implicitly solving the problem for $\mathcal{C}_{|w|}(\mathcal{Q}) \cong A\dgstab$.  Though the two models are closely related, they are not identical:  the bead model is a two-to-one covering of the arc model.  Each bead $B$ possesses a ``partner'' $\widetilde{B}$; if $B$ corresponds to the object $X \in A\dgstab$, then its partner corresponds to the isomorphic object $\Omega X(1)$.  The bead model thus provides a one-to-one correspondence with the indecomposable objects of the fractional Calabi-Yau category $A\grstab/\Omega^2(2)$.  Though we will not further investigate this phenomenon in this paper, we mention it as evidence that the extra symmetry present in the bead model is not merely a combinatorial artifact but rather captures actual structural information.

In Section \ref{Notation and Terminology}, we lay out notational conventions and definitions.  In Section \ref{Action}, we adapt the action of Chuang and Rouquier to negative Calabi-Yau categories.  In Section \ref{The Dg-stable Category} we review the basic properties of the dg-stable category of a symmetric algebra and show that it is a Calabi-Yau category.  In Section \ref{Bead Model}, we develop the combinatorial model for $A\dgstab$.  In Section \ref{mutations}, we lift the action of perverse equivalences to beads and prove transitivity.  In Section \ref{Connections}, we discuss connections with simple-minded systems and the arc model for $A\dgstab$.


\section{Notation and Terminology}
\label{Notation and Terminology}

All categories are assumed to be $k$-linear over a fixed algebraically closed field $k$.  All algebras are $k$-algebras.


\subsection{Graded Modules}

Let $A$ be a finite-dimensional graded algebra.  Let $A\grmod$ denote the category of finite-dimensional graded right $A$-modules.  Let $(n)$ denote the grading shift functor on $A\grmod$, given by $(X(n))^i = X^{n+i}$.  Let $A\grstab$ denote the stable category of graded modules.  The objects of $A\grstab$ are the objects of $A\grmod$; given $X,Y \in A\grstab$, define $\Hom_{A\grstab}(X,Y)$ to be the quotient of $\Hom_{A\grmod}(X,Y)$ by the $k$-subspace of all morphisms factoring through a graded projective module.  If $A$ is self-injective, $A\grstab$ admits the structure of a triangulated category, in which the syzygy functor $\Omega$, sending a module to the kernel of a projective cover, serves as the desuspension functor.


\subsection{Triangulated Categories}
Let $(\mathcal{T}, \Sigma)$ be a triangulated category with suspension functor $\Sigma$.  Let $S$ denote a collection of objects in $\mathcal{T}$.

For any morphism $f: X \rightarrow Y$, we write $C(f)$ for the object (unique up to non-canonical isomorphism) completing the triangle $X \xrightarrow{f} Y \rightarrow C(f) \rightarrow \Sigma X$.  We refer to $C(f)$ as the \textbf{cone} of $f$.

Let $S^\perp = \{ X \in Ob(\mathcal{T}) \mid \Hom(Y, X) = 0 \text{ for all } Y \in S\}$ and ${}^{\perp}S = \{ X \in Ob(\mathcal{T}) \mid \Hom(X, Y) = 0 \text{ for all } Y \in S\}$.

Let $\langle S \rangle$ denote the smallest full subcategory of $\mathcal{T}$ which contains $S$ and is closed under isomorphisms and extensions.

A \textbf{Serre functor} of $\mathcal{T}$ is an autoequivalence $\mathbb{S}$ of $\mathcal{T}$ such that there is an isomorphism $\Hom_{\mathcal{T}}(X, Y) \cong \Hom_{\mathcal{T}}(Y, \mathbb{S}X)^*$ which is natural in $X$ and $Y$.

$\mathcal{T}$ is said to be $w$-\textbf{Calabi-Yau} for some $w \in \mathbb{Z}$ if $\Sigma^w$ is a Serre functor for $\mathcal{T}$.


\subsection{$|w|$-Bases}

Let $(\mathcal{T}, \Sigma)$ be $w$-Calabi-Yau, for some $w < 0$.

Following Coelho Sim\~oes and Pauksztello \cite{simoes2018simple}, a tuple $(X_1, \cdots X_n)$ of objects in $\mathcal{T}$ called $|w|$-\textbf{orthogonal} if $dim\Hom(X_i, \Sigma^{-m}X_j) = \delta_{i=j}\delta_{m=0}$ for all $1 \le i, j \le n$ and $0 \le m \le |w|-1$.

If, in addition, we have that $\mathcal{T} = $$\langle \{\Sigma^{-m}X_i \mid 1 \le i \le n, 0 \le m < |w|\} \rangle$, we say $(X_1, \cdots, X_n)$ is a $|w|$-\textbf{basis} for $\mathcal{T}$.
 

\subsection{Maximal Extensions}
\label{Maximal Extensions}

Let $(\mathcal{T}, \Sigma)$ be a triangulated category and $\mathcal{S}$ be a collection of objects in $\mathcal{T}$.  In \cite{chuang2017perverse}, Chuang and Rouquier define maximal extensions as follows:

\begin{definition}{(Chuang, Rouquier, \cite{chuang2017perverse}, Definition 3.28)}
\label{maximal extension}
Let $f: X \rightarrow Y$ be a morphism in $\mathcal{T}$.

$f$ (or $X$) is a \textbf{maximal extension of $Y$ by $\mathcal{S}$} if $\Sigma^{-1} C(f) \in \mathcal{S}$ and $\Hom(\Sigma^{-1}C(f), S) \xrightarrow{\sim} \Hom(\Sigma^{-1}Y, S)$ is an isomorphism for all $S \in \mathcal{S}$.

$f$ (or $Y$) is a \textbf{maximal $\mathcal{S}$-extension by $X$} if $C(f) \in \mathcal{S}$ and $\Hom(S, C(f))$ $\xrightarrow{\sim} \Hom(S, \Sigma X)$ is an isomorphism for all $S \in \mathcal{S}$.
\end{definition}

If $X \in {}^{\perp}\mathcal{S} \cap \mathcal{S}^{\perp}$, Chuang and Rouquier (\cite{chuang2017perverse}, Lemma 3.29) prove that both maximal extensions of $\mathcal{S}$ by $X$ and maximal $X$-extensions of $\mathcal{S}$ are unique up to unique isomorphism (if they exist).  They also prove the following characterization of maximal extensions:

\begin{proposition}{(Chuang, Rouquier, \cite{chuang2017perverse}, Lemma 3.30)}
\label{orthogonal extensions}
Suppose $\mathcal{S}$ is closed under extensions.  Let $f: X \rightarrow Y$ be a morphism in $\mathcal{T}$.

Let $\Hom(Y, \mathcal{S}) = 0$.  Then $f$ is a maximal extension of $Y$ by $\mathcal{S}$ if and only if $\Sigma^{-1}C(f) \in \mathcal{S}$ and $\Hom(X, \mathcal{S}) = \Hom(X, \Sigma \mathcal{S}) = 0$.

Let $\Hom(\mathcal{S}, X) = 0$.  Then $f$ is a maximal $\mathcal{S}$-extension by $X$ if and only if $C(f) \in \mathcal{S}$ and $\Hom(\mathcal{S}, Y) = \Hom(\mathcal{S}, \Sigma Y) = 0$.
\end{proposition}

Suppose that $\mathcal{T}$ is $w$-Calabi-Yau for some $w < 0$, and fix an integer $n > 0$.  We say that $\mathcal{T}$ \textbf{admits $|w|$-orthogonal maximal extensions} if, given a $|w|$-orthogonal $n$-tuple $(X_1, \cdots, X_n)$, a subset $\mathcal{S} \subset \{X_i\}$, and $X_j \notin \mathcal{S}$, both the maximal extension of $X_j$ by $\mathcal{S}$ and the maximal $\mathcal{S}$-extension by $X_j$ exist.  We will generally ignore the dependence of this definition on the integer $n$, since we will be working with a fixed $n$ throughout the paper.








\subsection{Rooted Plane Trees}
\label{Rooted Plane Trees}

A \textbf{tree} is a connected graph without cycles.  We write $V_T$ and $E_T$ for the sets of vertices and edges, respectively, in $T$; the subscripts will be omitted when there is no risk of confusion.  A \textbf{rooted tree} is a pair $(T, r)$ where $r \in V_T$; $r$ is called the \textbf{root} of $T$.  Each vertex $v$ admits a unique minimal path $\gamma_v$ to the root; the \textbf{depth}, $d(v)$, of $v$ is the number of edges in this path.  If a vertex $u \neq v$ lies on $\gamma_v$, we say that $u$ is an \textbf{ancestor} of $v$ and that $v$ is a \textbf{descendant} of $u$.  Each vertex $v \neq r$ has a unique adjacent ancestor, called the \textbf{parent} of $v$, denoted $p(v)$.  For any vertex $v$, we say $u$ is a \textbf{child} of $v$ if $v$ is the parent of $u$, and we denote by $c(v)$ the set of children of $v$.  We say $v$ is a \textbf{leaf} if $v$ has no children.

For a finite rooted tree $(T, r)$, we define the \textbf{weight} $W(v)$ of $v$ to be the number of vertices of the subtree consisting of $v$ and its descendants; thus $W(r) = |V_T|$ and $W(v) = 1$ if and only if $v$ is a leaf.  It is clear that $W(v)$ is given recursively by $W(v) = 1 + \sum_{u \in c(v)} W(u)$.  There is a one-to-one correspondence between the edges of $(T, r)$ and the non-root vertices, with each edge corresponding to the incident vertex of greater depth.  Using this bijection, we define the \textbf{weight} of an edge $W(e)$ to be the weight of the corresponding vertex.

When we wish to emphasize the dependence on a choice of root, we will write $d_r(v), p_r(v), W_r(v)$, etc.

Let $T$ be a tree with $n$ edges.  We say a rooted tree $(T,r)$ is \textbf{balanced} if for all $v \in c(r)$ (or, equivalently, for all $v \neq r$), $W(v) \le \lfloor \frac{n+1}{2} \rfloor$.  If $v$ is a child of $r$ in $(T,r)$, we say the rooted tree $(T,v)$ is a \textbf{rebalancing of $(T,r)$ in the direction of $v$}.

A \textbf{plane tree} is a tree $T$ together with a cyclic ordering of the edges incident to each vertex.  One can specify this data by drawing $T$ in the plane such that each vertex is locally embedded.

Let $\mathcal{T}_n$ denote the set of (isomorphism classes of) trees with $n$ edges.  Let $\mathcal{PT}_n$ denote the set of (isomorphism classes of) plane trees with $n$ edges.

\begin{example}
Let $T$ be a line with $n$ edges.  If $n$ is even, then $(T, r)$ is balanced if and only if $r$ is the middle vertex of $T$.  If $n$ is odd, then $(T, r)$ is balanced if and only if $r$ is either of the vertices incident to the middle edge of $T$.
\end{example}

\begin{proposition}
\label{tree balancing}
Let $T$ be a tree with $n$ edges.  Then there exists $r \in V_T$ such that $(T, r)$ is balanced.  Either $T$ has a unique balancing root, or it has exactly two balancing roots $r$ and $r'$.  In the latter case, $n$ is odd, $r$ and $r'$ are adjacent and the edge joining them has weight $\frac{n+1}{2}$.  Conversely, if $n$ is odd and the rooted tree $(T, r)$ has an edge $r-r'$ of weight $\frac{n+1}{2}$, then $r$ and $r'$ are both balancing roots of $T$.
\end{proposition}

\begin{proof}
Choose a vertex $r\in V_T$ such that the quantity $max\{W_r(w) \mid w \in c_r(r)\}$ is minimized.  Suppose that $(T,r)$ is not balanced.  Choose the (necessarily unique) vertex $v \in c_r(r)$ such that $W_r(v) > \lfloor \frac{n+1}{2} \rfloor$, and consider the tree $(T,v)$.  We will show that $W_v(w) < W_r(v) = max\{W_r(w) \mid w \in c_r(r)\}$ for all $w \in c_v(v)$, which will contradict the minimality of $r$.

For all $w \in c_v(v)-\{r\}$, we have that $W_v(w) = W_r(w) < W_r(v)$.  Finally,
\begin{align*}
W_v(r) & = 1 + \sum_{w \in c_v(r)} W_v(w)  = 1 + \sum_{w \in c_r(r)- \{v\}} W_r(w) = n+1 - W_r(v)\\
& < n+1 - \lfloor \frac{n+1}{2} \rfloor = \lceil \frac{n+1}{2} \rceil
\end{align*}
hence $W_v(r) \le \lfloor \frac{n+1}{2} \rfloor < W_r(v)$.  We have obtained our contradiction, thus $(T, r)$ is balanced.

Next, suppose $(T, r)$ and $(T, r')$ are balanced, with $r \neq r'$.  Let $v$ be the parent of $r'$ with respect to $(T, r)$.  Then, since $(T, r')$ is balanced, we have that
\begin{align*}
W_r(r') & = 1 + \sum_{w \in c_r(r')} W_r(w)  = 1 + \sum_{w \in c_{r'}(r')- \{v\}} W_{r'}(w) = n+1 - W_{r'}(v)\\
& \ge n+1 - \lfloor \frac{n+1}{2} \rfloor = \lceil \frac{n+1}{2} \rceil
\end{align*}
But since $(T, r)$ is balanced, we have that $W_r(r') \le \lfloor \frac{n+1}{2} \rfloor \le \lceil \frac{n+1}{2} \rceil$.  Thus $\lfloor \frac{n+1}{2} \rfloor = \lceil \frac{n+1}{2} \rceil$, hence $n$ is odd.  Furthermore, $W_r(v) > W_r(r') = \frac{n+1}{2}$, which implies that $v = r$.  Thus $r$ and $r'$ are adjacent, and the edge between them has weight $W_r(r') = \frac{n+1}{2}$.  Since the total weight of the children of $r$ is $n$, there can be no other children of weight $\frac{n+1}{2}$, hence $r$ and $r'$ are the only balancing vertices of $T$.

For the final statement, suppose $r$ and $r'$ are adjacent vertices in $T$ such that $W_r(r') = \frac{n+1}{2}$.  The other children of $r$ have weight at most $n - \frac{n+1}{2} = \frac{n-1}{2}$, hence $(T, r)$ is balanced.  We have already seen that $W_{r'}(r) = n+1 - W_{r}(r') = \frac{n+1}{2}$; a symmetric argument then shows that $(T, r')$ is also balanced.
\end{proof}

\begin{remark}
One can find the balancing root(s) of a tree $T$ via a simple algorithm:  Pick an arbitrary vertex $r$ as the root.  If the tree is not balanced, rebalance the tree in the (unique) direction of the highest weighted child of $r$, until the tree is balanced.  If the balancing root has an incident edge of weight $\frac{n+1}{2}$, then both vertices incident to this edge are balancing roots.
\end{remark}


\section{The Action of Perverse Equivalences on a Category Admitting Minimal Extensions}
\label{Action}

\subsection{Basic Definitions}

Let $(\mathcal{T}, \Sigma)$ be a $k$-linear, Hom-finite, $w$-Calabi-Yau triangulated category, for some integer $w< 0$.  We also assume that $\mathcal{T}$ is Krull-Schmidt, i.e. every object in $\mathcal{T}$ is isomorphic to a direct sum of objects with local endomorphism rings.  Suppose that $\mathcal{T}$ admits $|w|$-orthogonal maximal extensions.  We fix a positive integer $n$.

\begin{definition}
\label{simple-minded tuple}
Let $\widehat{\mathcal{E}}$ be the set of all $|w|$-orthogonal $n$-tuples of objects of $\mathcal{T}$ (up to isomorphism).  Let $\mathcal{E}$ be the subset of all $n$-tuples which form a $|w|$-basis.

We shall refer to elements of $\widehat{\mathcal{E}}$ as \textbf{orthogonal tuples}.  Elements of $\mathcal{E}$ will be referred to as \textbf{generating tuples}.
\end{definition}

Note that if $(X_i)_i \in \mathcal{E}$ (resp. $\widehat{\mathcal{E}}$) then $(\Sigma^mX_{\sigma(i)})_i \in \mathcal{E}$ (resp. $\widehat{\mathcal{E}}$) for any $m \in \mathbb{Z}$ and any $\sigma \in \mathfrak{S}_n$.  We define an equivalence relation $\sim$ on $\mathcal{E}$ (resp. $\widehat{\mathcal{E}}$) by $(X_i)_i \sim (Y_i)_i$ if there exists $m \in \mathbb{Z}, \sigma \in \mathfrak{S}_n$ such that $Y_i = \Sigma^m X_{\sigma(i)}$ for each $i$.  Since we are interested in classifying and counting the members of $\mathcal{E}$, it will frequently be helpful to work modulo these symmetries.

\begin{remark}
In practice, the triangulated category $\mathcal{T}$ will usually arise from some category of modules over an algebra $A$; in this case, the number of simple $A$-modules is a natural choice for $n$.  For this reason, we suppress the dependence of $\mathcal{E}$ on the choice of $n$ in our notation.
\end{remark}

We now introduce some terminology that will be convenient throughout the rest of this paper:

\begin{definition}
\label{elementary}
Let $X \in Ob(\mathcal{T})$.  We say $X$ is \textbf{elementary} if
\begin{equation*}
dim \Hom(X, \Sigma^{-m}X) = \delta_{0 = m}
\end{equation*}
 for all $0 \le m < |w|$.
\end{definition}

\begin{definition}
Let $X, Y$ be elementary objects of $\mathcal{T}$.  We say $X$ and $Y$ are \textbf{independent} if
\begin{equation*}
dim\Hom(X, \Sigma^{-m}Y) = dim\Hom(Y, \Sigma^{-m}X) = \delta_{0 = m}\delta_{X \cong Y}
\end{equation*}
for all $0 \le m < |w|$.
\end{definition}

Thus an orthogonal tuple is a tuple of distinct elementary objects which are pairwise independent.

Let $\mathcal{P}'(n)$ denote the set of proper subsets of $[n] := \{1, \cdots n\}$.  The symmetric group $\mathfrak{S}_n$ acts on $\mathcal{P}'(n)$ in the obvious way.

\begin{definition}
\label{generator action}
Define an action of $\Xi := Free(\mathcal{P}'(n)) \rtimes \mathfrak{S}_n$ on $\mathcal{E}, \widehat{\mathcal{E}}$ as follows:\\
1)  $\mathfrak{S}_n$ acts on $(X_i)_i \in \widehat{\mathcal{E}}$ by permutation of indices.\\
2)  Given $S \in \mathcal{P}'(n)$, $(X_i)_i \in \widehat{\mathcal{E}}$, define $S\cdot (X_i)_i = (X_i')_i$ by\\
\begin{equation*}
X_i' = \begin{cases}
X_i & i \in S\\
\Sigma^{-1}(X_i)_S & i \notin S
\end{cases}
\end{equation*}
where $(X_i)_S$ is the minimal extension of $X_i$ by $\mathcal{S} := \langle X_s \mid s \in S \rangle$.\\
3)  Define $S^{-1}\cdot (X_i)_i = (X_i')_i$ by\\
\begin{equation*}
X_i' = \begin{cases}
X_i & i \in S\\
\Sigma(X_i)^S & i \notin S
\end{cases}
\end{equation*}
where $(X_i)^S$ is the minimal $\mathcal{S}$-extension by $X_i$.
\end{definition}

Rouquier and Chuang \cite{chuang2017perverse} defined the above action on tilting complexes in the derived category of a finite-dimensional symmetric algebra (Section 5.2), as well as for bases of (-1)-Calabi-Yau categories (Section 7).  

We must show that this action is well-defined on both sets.

\begin{proposition}
\label{action justification}
The action of $\Xi$ on $\mathcal{\widehat{E}}$ is well-defined.
\end{proposition}

\begin{proof}
It suffices to show that the action of $Free(\mathcal{P}'(n))$ on $\widehat{\mathcal{E}}$ is well-defined.

Take $(X_i)_i \in \widehat{\mathcal{E}}$, $S \subsetneq [n]$.  
By assumption, the minimal extension $(X_j)_S$ exists for each $j \notin S$; we now verify that the tuple $S\cdot (X_i)_i$ is $|w|$-orthogonal.  Let $\mathcal{S}= \langle \{X_i \mid i \in S\} \rangle$.  

Fix $j \notin S$.  Let $f_j:  (X_j)_S \rightarrow X_j$ be the morphism defining the extension.  Let $Y \in \mathcal{S}$.

We claim that, for all $0 \le m \le |w|$,
\begin{equation}
\label{aj1}
\Hom(\Sigma^{-1}(X_j)_S, \Sigma^{-m}Y) = 0
\end{equation}
By Proposition \ref{orthogonal extensions}, Equation (\ref{aj1}) holds for $m = 0, 1$.  For $2 \le m \le |w|$, apply $\Hom(-, \Sigma^{-m+1}Y)$ to the triangle $\Sigma^{-1}C(f_j)\rightarrow (X_j)_S \xrightarrow{f_j} X_j \rightarrow C(f_j)$.  Since $(X_i)_i$ is $|w|$-orthogonal, $\Hom(X_j, \Sigma^{-m+1}Y) = 0$.  Similarly, $\Hom(\Sigma^{-1}C(f_j), \Sigma^{-m+1}Y) = 0$, since $\Sigma^{-1}C(f_j) \in \mathcal{S}$.  It follows that $\Hom((X_j)_S, \Sigma^{-m+1}Y) = 0$, hence $\Hom(\Sigma^{-1}(X_j)_S, \Sigma^{-m}Y)=0$.

Next, we claim that, for all $0 \le m \le |w|$,
\begin{equation}
\label{aj2}
\Hom(Y, \Sigma^{-m-1}(X_j)_S) =0
\end{equation}
By Serre duality and Equation (\ref{aj1}),
\begin{align*}
\Hom(Y, \Sigma^{-m-1}(X_j)_S) & \cong \Hom(\Sigma^{-m-1}(X_j)_S, \Sigma^{w}Y)^*\\
& \cong \Hom(\Sigma^{-1}(X_j)_S, \Sigma^{m+w}Y)^*\\
& = 0
\end{align*}
for all $0 \le m \le |w|$.


Let $j, l\notin S$.  We claim that, for all $0 \le m \le |w|-1$,
\begin{equation}
\label{aj4}
\Hom(\Sigma^{-1}(X_l)_S, \Sigma^{-m-1}(X_j)_S)\cong \Hom(\Sigma^{-1}X_l, \Sigma^{-m-1}X_j)
\end{equation}
Apply $\Hom(\Sigma^{-1}(X_l)_S, -)$ to the triangle $\Sigma^{-m-2}C(f_j) \rightarrow \Sigma^{-m-1}(X_j)_S \rightarrow \Sigma^{-m-1}X_j \rightarrow \Sigma^{-m-1} C(f_j)$.  $\Sigma^{-1}C(f_j) \in \mathcal{S}$, so by (\ref{aj1}), we have for all $0 \le m \le |w|-1$,
\begin{equation*}
\Hom(\Sigma^{-1}(X_l)_S, \Sigma^{-m-2}C(f_j)) = 0 = \Hom(\Sigma^{-1}(X_l)_S, \Sigma^{-m-1}C(f_j))
\end{equation*}
Thus $\Hom(\Sigma^{-1}(X_l)_S, \Sigma^{-m-1}(X_j)_S) \cong \Hom(\Sigma^{-1}(X_l)_S, \Sigma^{-m-1}X_j)$.  Next, apply $\Hom(-, \Sigma^{-m-1}X_j)$ to the triangle $\Sigma^{-2}C(f_l) \rightarrow \Sigma^{-1}(X_l)_S \rightarrow \Sigma^{-1}X_l$ $\rightarrow \Sigma^{-1}C(f_l)$.  $\Sigma^{-1}C(f_l) \in \mathcal{S}$, so for all $0 \le m \le |w|-1$,
\begin{equation*}
\Hom(\Sigma^{-2}C(f_l), \Sigma^{-m-1}X_j) = 0 = \Hom(\Sigma^{-1}C(f_l), \Sigma^{-m-1}X_j)
\end{equation*}
Thus $\Hom(\Sigma^{-1}(X_l)_S, \Sigma^{-m-1}X_j) \cong \Hom(\Sigma^{-1}X_l, \Sigma^{-m-1}X_j)$.  Combining the two isomorphisms, we obtain the desired equality.

Substituting $j = l$ into Equation (\ref{aj4}) and using the fact that $X_j$ is elementary, we have that $\Sigma^{-1}(X_j)_S$ is elementary.  When $l \neq j$, independence of $\Sigma^{-1}(X_l)_S$ and $\Sigma^{-1}(X_j)_S$ follows from Equation (\ref{aj4}) and the independence of $X_l$ and $X_j$.  When $j \notin S$ and $l \in S$, independence of $\Sigma^{-1}(X_j)_S$ and $X_l$ follows from Equations (\ref{aj1}) and (\ref{aj2}).  Thus $\widehat{\mathcal{E}}$ is closed under the action of $S \subsetneq [n]$.

The proof that $\widehat{\mathcal{E}}$ is closed under the action of $S^{-1}$ is dual.

Finally, we must show that the action of $S$ and $S^{-1}$ are mutually inverse.  To show $S^{-1}S \cdot (X_i)_i = (X_i)_i$, it is enough to show that for each $j\notin S$, $(\Sigma^{-1}(X_j)_S)^{S} \cong \Sigma^{-1}X_j$.  It is easy to verify that the map $\Sigma^{-1}(X_j)_S \rightarrow \Sigma^{-1}X_j$ satisfies the conditions of Proposition \ref{orthogonal extensions}, hence $\Sigma^{-1}X_j$ is isomorphic to $(\Sigma^{-1}(X_j)_S)^{S}$.  The proof that $SS^{-1}\cdot (X_i) = (X_i)$ is dual.
\end{proof}

\subsection{Filtrations}

In the previous section, we proved that $\widehat{\mathcal{E}}$ is stable under the action of $\Xi$.  In this section, we show that the subset $\mathcal{E}$ is stable under this action.  To accomplish this, we will need a few technical results.

\begin{definition}
Let $\mathcal{F}$ be a $|w|$-basis for $\mathcal{T}$.  Let $M \in \mathcal{T}$.  A \textbf{descending $\mathcal{F}$-filtration} is a sequence of morphisms $f_i:  M_i \rightarrow M_{i-1}$, $1 \le i \le m$, with $M_m = M$, $M_0 = 0$ and $\Sigma^{-1}C(f_i) = \Sigma^{-d_i}S_i$ for some $S_i \in \mathcal{F}$ and $1 \le d_i < |w|$.  We say this filtration is \textbf{nice} if the sequence $\{d_i\}$ is non-strictly decreasing.

Dually, define an \textbf{ascending $\mathcal{F}$-filtration} to be a sequence of morphisms $f_i:M_{i-1}$ $\rightarrow M_{i}$, $1 \le i \le m$, with $M_0 = 0$, $M_m = M$, and $C(f_i) = \Sigma^{-d_i}S_i$ for some $S_i \in \mathcal{F}$.  We say this filtration is \textbf{nice} if the sequence $\{d_i\}$ is non-strictly increasing.  For both filtrations, we shall call $m$ the \textbf{length} of the filtration.

Given a descending (resp., ascending) filtration of $M$, we shall write $M = [\Sigma^{-d_m}S_m, \cdots \Sigma^{-d_1}S_1]_d$ (resp., $M = [\Sigma^{-d_1}S_1, \cdots \Sigma^{-d_m}S_m]_a$).  If $m$ is minimal, we shall refer to $m$ as the descending (resp., ascending) \textbf{length} of $M$, which we shall denote by $l_d(M)$, (resp., $l_a(M)$).  We shall drop the modifiers and subscripts when there is no risk of confusion and refer simply to ``lengths'' and ``filtrations''.

We shall refer to the $\Sigma^{-d_i}S_i$ as the \textbf{factors} of $M$.  If a factor appears as the right-most (resp., left-most) term in a nice, minimal descending (resp., ascending) filtration of $M$, we say that factor \textbf{lies in the head} (resp., \textbf{socle}), of $M$.
\end{definition}

Intuitively, filtrations provide a triangulated analogue of composition series.  An object may have many different filtrations relative to a given basis, but filtrations of minimal length are relatively well-behaved.  The following lemma is adapted from \cite{chuang2017perverse}, Lemma 7.1.

\begin{lemma}
\label{filtration}
Let $\mathcal{F}$ be a $|w|$-basis for $\mathcal{T}$.  Let $M \in \mathcal{T}$.  Then:\\
1)  $M$ has a descending $\mathcal{F}$-filtration $M = [\Sigma^{-d_m}S_m, \cdots, \Sigma^{-d_1}S_1]$ which is both nice and of minimal length.  Given any minimal filtration of $M$, there is a nice, minimal filtration of with the same multiset of factors.\\
2)  Any two minimal filtrations of $M$ have the same multiset of factors.\\
3)  Using the notation of part 1), if $\Hom(M, \Sigma^{-d_1}S) \neq 0$ for some $S \in \mathcal{F}$, then $\Sigma^{-d_1}S$ is isomorphic to one of the factors of $M$, and $\Sigma^{-d_1}S$ lies in the head of $M$.\\
4)  For any nice, minimal descending filtration, $M = [\Sigma^{-d_m}S_m, \cdots, \Sigma^{-d_1}S_1]$, the composition $M = M_m \rightarrow \cdots \rightarrow M_1 \cong \Sigma^{-d_1}S_1$ is nonzero.\\
The dual statements hold for ascending filtrations.
\end{lemma}

\begin{proof}
For 1), since $\mathcal{F}$ is a $|w|$-basis, every object of $\mathcal{T}$ has a finite $\mathcal{F}$-filtration, hence a minimal one.  Let $M = [\Sigma^{-d_m}S_m, \cdots, \Sigma^{-d_1}S_1]$ be one such minimal filtration.  If this filtration is not nice, there exists $i$ such that $d_i > d_{i-1}$.  Consider the following diagram, obtained from the octahedron axiom:
\begin{eqnarray*}
\begin{tikzcd}
M_{i} \arrow[rr, "f_{i}"] &[-20pt] & [-20pt] M_{i-1} \arrow[rr, "f_{i-1}"] \arrow[dl] & [-20pt] & [-20pt] M_{i-2} \arrow[dl] \arrow[ddll, bend left]\\
& \Sigma^{-d_i + 1}S_i \arrow[ul, dashed] \arrow[dr] & & \Sigma^{-d_{i-1} + 1}S_{i-1} \arrow[ll, "\phi", dashed, swap] \arrow[ul, dashed]&\\
& & C(f_{i-1} f_{i}) \arrow[ur] \arrow[uull, dashed, bend left] & &
\end{tikzcd}
\end{eqnarray*}

Since $d_i > d_{i-1}$, the morphism $\phi: \Sigma^{-d_{i-1}+1}S_{i-1} \rightarrow \Sigma^{-d_i + 2} S_i$ is either zero or an isomorphism.  If $\phi$ is an isomorphism, then $C(f_{i-1} f_{i}) = 0$, hence $f_{i-1}f_{i}$ is an isomorphism.  But then $M_{i}$ and $M_{i-1}$ can be deleted from the filtration, since the composite map $f_{i-1}f_if_{i+1}: M_{i+1} \rightarrow M_{i-2}$ has cone isomorphic to $C(f_{i+1})$.  (If $i=m$, one simply deletes the last two terms, since $M_{m-2} \cong M$.)  This contradicts minimality of $m$, hence we must have $\phi = 0$.

Since $\phi = 0$, $C(f_{i-1} f_{i}) \cong \Sigma^{-d_i + 1}S_i \oplus \Sigma^{-d_{i-1} + 1}S_{i-1}$.  Let $X$ be the cone of the composition $g: \Sigma^{-d_{i-1}}S_{i-1} \rightarrow \Sigma^{-1}C(f_{i-1} f_{i}) \rightarrow M_{i}$.  Let $f_{i}': M_{i} \rightarrow X$ be the natural map into the cone.  By construction, $\Sigma^{-1}C(f_{i}') = \Sigma^{-d_{i-1}}S_{i-1}$.  Furthermore, applying the octahedron axiom to $g$ yields a map $f_{i-1}': X \rightarrow M_{i-2}$ such that $\Sigma^{-1}C(f_{i-1}') = \Sigma^{-d_{i}}S_{i}$.  Thus replacing $f_{i}$ and $f_{i-1}$ with $f_{i}'$ and $f_{i-1}'$ yields a minimal filtration with $\Sigma^{-d_{i}}S_{i}$ and $\Sigma^{-d_{i-1}}S_{i-1}$ swapped.  We may repeat this process until there are no more inversions, yielding a nice, minimal filtration.  Since the factors have only been permuted, the multiset of factors remains unchanged.

For 3), let $1 \le r \le m$ be minimal such that there is a nonzero morphism $M_r \rightarrow \Sigma^{-d_1}S$.  Consider the triangle $\Sigma^{-d_r}S_r \rightarrow M_r \rightarrow M_{r-1} \rightarrow \Sigma^{-d_r +1}S_r $.  Since $\Hom(M_{r-1}, \Sigma^{-d_1}S) = 0$, the nonzero space $\Hom(M_r, \Sigma^{-d_1}S)$ injects into $\Hom(\Sigma^{-d_r}S_r, \Sigma^{-d_1}S)$.  Since the filtration is nice, $d_1 \ge d_r$; since the Hom space is nonzero, we deduce that $d_r = d_1$ and $S_r \cong S$.  It follows that the composition $\Sigma^{-d_r}S_r \rightarrow M_r \rightarrow \Sigma^{-d_1}S$ is an isomorphism, hence the above triangle splits and $M_r \cong \Sigma^{-d_1}S \oplus M_{r-1}$.

Define a new filtration of $M$ as follows.  For $1 \le i \le r-1$, let $M_i' = \Sigma^{-d_1}S \oplus M_{i-1}$ and let $f_i': M_i' \rightarrow M_{i-1}'$ be the direct sum of the identity map and $f_{i-1}$.  Since $M_r \cong \Sigma^{-d_1}S \oplus M_{r-1}$, we can define $f_r'$ in the same way.  Let all remaining objects and maps remain the same.  It is straightforward to verify that this is a filtration identical to the original, except that the last $r$ factors have been cyclically permuted, so that $\Sigma^{-d_r}S_r \cong \Sigma^{-d_1}S$ is the final term.  It is clear that this filtration is nice and minimal, hence $\Sigma^{-d_1}S$ lies in the head of $M$.

To prove 4), for each $1 \le i \le m$ let $g_i:  M_i \rightarrow M_1 \cong \Sigma^{-d_1}S_1$ be the natural composition.  Suppose for a contradiction that $g_m = 0$ and let $i \ge 2$ be minimal such that $g_i = 0$.  Decompose $g_i$ as $M_i \xrightarrow{g} M_2 \xrightarrow{f_2} M_1$ and apply the octahedron axiom.  We obtain a triangle $\Sigma^{-1}C(g) \rightarrow M_i \oplus \Sigma^{d_1-1}S_1 \rightarrow \Sigma^{-d_2}S_2 \rightarrow C(g)$.  We have that $\Sigma^{-1} C(g) = [\Sigma^{-d_i}S_i, \cdots \Sigma^{-d_3}S_3]$; by niceness of the filtration it follows that $\Hom(\Sigma^{-1}C(g), \Sigma^{-d_1-1}S_1) = 0$.  Therefore the morphism $\Sigma^{-1}C(g) \rightarrow M_i \oplus \Sigma^{d_1-1}S_1$ factors through the inclusion of $M_i$, hence the cone of this morphism is $M_2 \oplus \Sigma^{d_1-1}S_1 \cong \Sigma^{-d_2}S_2$.  This contradicts locality of $\End(S_2)$, thus $g_i \neq 0$ for all $i$.

The proof of 2) is by induction on the length, $m$, of $M$.  For $m = 0, 1$, the result is clear.  Suppose the result holds for all lengths less than $m$.  Given $M = [\Sigma^{-d_m}S_m, \cdots, \Sigma^{-d_1}S_1] = [\Sigma^{-d_m'}S_m', \cdots, \Sigma^{-d_1'}S_1']$ two minimal filtrations, by 1) we can rearrange the factors and assume WLOG that both filtrations are nice.  Then $\Hom(M, \Sigma^{-d}S) = 0$ for any $S \in \mathcal{F}$, $d_1 < d \le |w|-1$.  By 4), $\Hom(M, \Sigma^{-d_1'}S_1') \neq 0$, hence $d_1' \le d_1$.  A symmetric argument gives the reverse inequality, hence $d_1 = d_1'$.  By 3), $\Sigma^{-d_1}S_1 \cong \Sigma^{-d_r'}S_r'$ for some $r$, and we can rearrange the second filtration so that $\Sigma^{-d_1}S_1$ is the last term.  Since both filtrations now end in $\Sigma^{-d_1}S_1$, we obtain two nice, minimal filtrations of $M' = \Sigma^{-1}C(M \rightarrow \Sigma^{-d_1}S_1)$ whose factor multisets correspond to the original factor multisets with one copy of $\Sigma^{-d_1}S_1$ removed.  Applying the induction hypothesis to $M'$, we are done.
\end{proof}

The following technical lemma describes the interaction between filtrations and maximal extensions.

\begin{lemma}
\label{medium extension}
Let $\mathcal{F}$ be a $|w|$-basis for $\mathcal{T}$ and let $\mathcal{S} \subset \mathcal{F}$.  Let $T \in \mathcal{F} - \mathcal{S}$ and let $T_S \rightarrow T$ denote the maximal extension of $T$ by $\mathcal{S}$.  Suppose this map factors as $T_S \xrightarrow{f} N \xrightarrow{g} T$ for some object $N = [S_k, \cdots, S_2, T]_d$, with $S_i \in \mathcal{S}$.  Suppose that $\Hom(N, \mathcal{S}) = 0$.  Then $\Sigma^{-1} C(f) \in \langle \mathcal{S} \rangle$.

Dually, let $T \rightarrow T^S$ denote the maximal $\mathcal{S}$-extension by T.  Suppose this map factors as $T \xrightarrow{g} N \xrightarrow{f} T^S$ for some object $N = [T, S_2, \cdots, S_k]_a$, with $S_i \in \mathcal{S}$.  Suppose $\Hom(\mathcal{S}, N) = 0$.  Then $C(f) \in \langle \mathcal{S} \rangle$.
\end{lemma}

\begin{proof}
Applying the octahedron axiom to the composition $gf$, we obtain a triangle $\Sigma^{-2}C(g) \rightarrow \Sigma^{-1}C(f) \rightarrow \Sigma^{-1}C(gf) \rightarrow \Sigma^{-1}C(g)$, where $\Sigma^{-1}C(g)$, $\Sigma^{-1}C(gf) \in \langle \mathcal{S} \rangle$.  It follows that $\Sigma^{-1}C(f)$ has a nice, minimal filtration whose factors lie in $\mathcal{S} \cup \Sigma^{-1}\mathcal{S}$.  We have that $\Hom(\Sigma^{-1}N, \Sigma^{-1}\mathcal{S}) = 0 = \Hom(T_S, \Sigma^{-1}\mathcal{S})$, hence $\Hom(\Sigma^{-1}C(f), \Sigma^{-1}\mathcal{S}) = 0$.  It follows from Lemma \ref{filtration} that $\Sigma^{-1}C(f)$ can have no factors lying in $\Sigma^{-1}\mathcal{S}$.  Therefore $\Sigma^{-1}C(f) \in \langle \mathcal{S} \rangle$.

The proof of the second statement is dual.
\end{proof}





We are now ready to prove that $\mathcal{E}$ is closed under the action of $\Xi$.  The following result is based on \cite{chuang2017perverse}, Proposition 7.4.

\begin{theorem}
The action of $\Xi$ on $\mathcal{E}$ is well-defined.
\end{theorem}

\begin{proof}
We must show that the action of $S\subsetneq [n]$ on an orthogonal tuple preserves the property of being a generating tuple.  Let $(X_i)_i \in \mathcal{E}$, let $\mathcal{S} = \{X_i \mid i \in S\}$, let $\mathcal{F} = \{X_i\}$, and let $\mathcal{F}' = \{X_i'\}$.  Let $\mathcal{G} = \bigcup_{i= 0}^{|w|-1}\Sigma^{-i}\mathcal{F}$ and $\mathcal{G}' = \bigcup_{i= 0}^{|w|-1}\Sigma^{-i}\mathcal{F}'$.  Then $\langle \mathcal{G} \rangle = \mathcal{T}$, and we must show that the same holds for $\langle \mathcal{G}' \rangle$.

Take a nonzero object $M \in \mathcal{T}$.  We first consider the special case where no $\Sigma^{-i} Y$ lies in the head of $\Sigma M$, for any $Y \in \mathcal{S}, 0 \le i < |w|$.  We claim that $M \in \langle \mathcal{G}' \rangle$; the proof will be by induction on the $\mathcal{F}$-length, $m$, of $\Sigma M$.  If $m = 1$, then $\Sigma M \cong \Sigma^{-i}T$ for some $T \in \mathcal{F} - \mathcal{S}$.  We have a triangle $\Sigma^{-i-1}T_S \rightarrow M \rightarrow \Sigma^{-i}Y \rightarrow \Sigma^{-i}T_S$ for some $Y \in \langle \mathcal{S} \rangle$.  For any $0 \le i < |w|$, $\Sigma^{-i-1}T_S, \Sigma^{-i}Y \in \mathcal{G}'$, hence $M \in \langle \mathcal{G}' \rangle$.

Now suppose $m >1 $ and the result holds for lower lengths.  By Lemma \ref{filtration}, $\Sigma M$ must have a nice, minimal descending $\mathcal{G}$-filtration ending in some $\Sigma^{-d_1}T_1$, where $T_1 \in \mathcal{F}- \mathcal{S}, 0 \le d_1 < |w|$.  There exists a maximal $0 \le k \le m$ such that there exists a minimal filtration of the form
\begin{equation*}
\Sigma M = [\Sigma^{-d_m}T_m, \cdots, \Sigma^{-d_{k+1}}T_{k+1}, \Sigma^{-d_k}S_k, \cdots \Sigma^{-d_2}S_2, \Sigma^{-d_1}T_1]
\end{equation*}
where each $S_j \in \mathcal{S}$, and $T_{k+1} \in \mathcal{F} - \mathcal{S}$.  (If $k = m$, the filtration starts with $\Sigma^{-d_m}S_m$.)  The octahedron axiom gives us a triangle $\Sigma M' \rightarrow \Sigma M \rightarrow \Sigma M'' \rightarrow \Sigma^2M'$, where $\Sigma M'$ and $\Sigma M''$ have (necessarily minimal) filtrations given by $\Sigma M' = [\Sigma^{-d_m}T_m, \cdots, \Sigma^{-d_{k+1}}T_{k+1}]$ and $\Sigma M'' = [\Sigma^{-d_k}S_k, \cdots \Sigma^{-d_2}S_2, \Sigma^{-d_1}T_1]$.

Note that there is no minimal filtration of $\Sigma M'$ whose last factor is of the form $\Sigma^{-d_{k+1}}S_{k+1}$, with $S_{k+1} \in \mathcal{S}, 0 \le d_{k+1} < |w|$.  If so, we could concatenate this filtration of $\Sigma M'$ with the given filtration $\Sigma M''$ to produce a new minimal filtration for $\Sigma M$ which would contradict the maximality of $k$.  Since the length of $\Sigma M''$ is at least one, $\Sigma M'$ has length strictly shorter than $\Sigma M$.  By the induction hypothesis, $M' \in \langle \mathcal{G}' \rangle$.

We now show that $M'' \in \langle \mathcal{G}' \rangle$.  By the proof of part 1) of Lemma \ref{filtration}, by rearranging the $S_i$ we may assume WLOG that the filtration for $\Sigma M''$ expressed above is nice.  Let $1 \le r \le k$ be maximal such that $d_r = d_1$.  We may express $\Sigma M''$ as the triangle $\Sigma N_1 \rightarrow \Sigma M'' \rightarrow \Sigma N_2 \rightarrow \Sigma^2 N_1$, where $\Sigma N_1 = [\Sigma^{-d_k}S_k, \cdots \Sigma^{-d_{r+1}}S_{r+1}]$ and $\Sigma N_2 = \Sigma^{-d_1}([S_r, \cdots S_2, T_1])$ are nice, minimal filtrations.  For all $r < i \le k$, we have that $d_i < d_1 \le |w| - 1$, thus $N_1 = [\Sigma^{-d_k -1}S_k, \cdots \Sigma^{-d_{r+1}-1}S_{r+1}] \in \langle \mathcal{G}' \rangle$.

Next, $\Hom((T_1)_S, \Sigma \mathcal{S}) = 0$, hence the minimal extension $(T_1)_S \rightarrow T_1$ factors through the natural map $[S_r, \cdots S_2, T_1] \rightarrow T_1$.  Note also that $\Hom([S_r, \cdots S_2, T_1], \mathcal{S}) = 0$; otherwise by part 3) of Lemma \ref{filtration}, there would be some member of $\Sigma^{-d_1}\mathcal{S}$ lying in the head of $\Sigma N_2$.  But this is impossible, since any factor in the head of $\Sigma N_2$ also lies in the head of $M'$ and $M$, and the head of $M$ contains no such factors by assumption.  The hypotheses of Lemma \ref{medium extension} are satisfied, and so we obtain a triangle $\Sigma^{-1}(T_1)_S \rightarrow \Sigma^{-1}[S_r, \cdots S_2, T_1] \rightarrow Y \rightarrow (T_1)_S$, with $Y \in \langle \mathcal{S} \rangle$.  Applying $\Sigma^{-d_1}$ to this triangle, we obtain $\Sigma^{-d_1 -1} (T_1)_S \rightarrow N_2 \rightarrow \Sigma^{-d_1}Y$.  Since $0 \le |d_1| < |w|$, both of the outside terms lie in $\langle \mathcal{G}' \rangle$, hence so does $N_2$.  It follows immediately that $M''$ and therefore $M$ lie in $\langle \mathcal{G}' \rangle$.  This concludes our proof of the special case.

We are now ready to prove the general case; it suffices to show that $\mathcal{G} \subset \langle \mathcal{G}' \rangle$.  By definition, $\bigcup_{i=0}^{|w|-1}\Sigma^{-i}\mathcal{S} \subset \mathcal{G}'$.  For $T \notin \mathcal{S}, 0 < i \le |w| -1$, the triangle $\Sigma^{-i-1}C(f) \rightarrow \Sigma^{-i}T_S \rightarrow \Sigma^{-i}T \rightarrow \Sigma^{-i}C(f)$ shows that $\Sigma^{-i}T \in \langle \mathcal{G}' \rangle$.  It remains to show that $T \in \langle \mathcal{G}' \rangle$; we shall reduce this problem to the special case shown above.

Note that $\Hom(\Sigma^i Y, T) = 0$ for all $Y \in \mathcal{S}, 0 \le i \le |w|-1$, hence $\Hom(\Sigma T, \Sigma^{-|w|+1 + i}Y) = 0$ by Serre duality.  In particular, by Lemma \ref{filtration}, part 4), $\Sigma T$ has no nice, minimal descending filtration ending in $\Sigma^{i}Y$, for any $Y \in \mathcal{S}$, $0 \le i \le |w|-1$.  By the special case, $T \in \langle \mathcal{G}' \rangle$, and we are done.
\end{proof}


\section{The Dg-stable Category}
\label{The Dg-stable Category}

Let $A$ be a finite-dimensional, self-injective graded $k$-algebra, with $A^{>0} = 0$.  We view $A$ as a dg-module with zero differential.  Let $D^b_{dg}(A)$ be the bounded derived category of finite-dimensional differential graded right $A$-modules, and let $D^{perf}_{dg}(A)$ be the thick subcategory generated by $A$.  We define the \textbf{differential graded stable category} of $A$ to be $A\dgstab:= D^b_{dg}(A)/D^{perf}_{dg}(A)$.  The basic properties of this triangulated category are discussed in \cite{brightbill2018differential}.  We state the main result, which links the dg-stable category to the orbit category $A\grstab/ \Omega(1)$.  (For more about orbit categories, see Keller \cite{Keller2005}.)

\begin{theorem}{(\cite{brightbill2018differential}, Theorem 3.10)}

There is a fully faithful functor $F_A:  A\grstab/\Omega(1) \rightarrow A\dgstab$, which is the identity on objects.  The image of $F_A$ generates $A\dgstab$ as a triangulated category.
\end{theorem}

Dg-stable categories provide many examples of negative Calabi-Yau categories.

\begin{proposition}
\label{Calabi-Yau}
Let $A$ be a finite-dimensional symmetric algebra with a non-positive grading.  Suppose the socle of $A$ is concentrated in degree $-d$ for some $d \ge 0$.  Then $A\dgstab$ is $-(d+1)$-Calabi-Yau.
\end{proposition}

\begin{proof}
For any self-injective algebra $A$, $A\stab$ has Serre functor $\mathbb{S} = \Omega \circ \nu$, where $\nu := A^* \otimes_A -$ is the Nakayama functor.  (See for instance \cite{erdmann2006stable}.)  Though we are unable to find a reference, it is well-known that the same result holds for $A\grstab$ when $A$ is a graded self-injective algebra.  Since $A$ is symmetric with socle concentrated in degree $-d$, we have that $\nu \cong (-d)$.  It follows that $\mathbb{S}$ is a Serre functor on the orbit category $A\grmod/\Omega(1)$.  Furthermore, since $\Omega \cong (-1)$ in this orbit category, we have that $\mathbb{S} \cong (-d-1)$ in $A\grmod/\Omega(1)$.  Since $A\grstab/\Omega(1)$ generates $A\dgstab$ as a triangulated category,  it follows that $(-d-1)$ extends to a Serre functor on all of $A\dgstab$.
\end{proof}


\section{A Combinatorial Model for $A\dgstab$}
\label{Bead Model}

\subsection{The Category}
In this section, we shall study the action of perverse equivalences on a specific triangulated category, namely the dg-stable category of a Brauer tree algebra.  The structure of this category was studied in detail in \cite{brightbill2018differential}, Section 6; we shall work in the same setting and use the same notational conventions, which we summarize below.

Let $k$ be an algebraically closed field.  Let $n \ge 2, d \ge 0$ be integers.  Let $A$ be the graded Brauer tree algebra on the star with $n$ edges, with socle in degree $-d$; this determines $A$ up to graded Morita equivalence.  We let $S_1, \cdots S_n$ denote the $n$ simple $A$-modules graded in degree $0$, and we let $M^i_j$ denote the indecomposable module with head $S_i$ and socle a shift of $S_j$.  The grading is chosen so that $dim\Ext^1(S_i, S_{i+1}) = 1$ for $1 \le i < n$ and $dim\Ext^1(S_n, S_1(d)) = 1$.  For more information on Brauer tree algebras, we refer the reader to Schroll \cite{schroll2018brauer}.

$A$ is a self-injective Nakayama algebra, so by \cite{brightbill2018differential}, Lemma 4.6, we have that $A\dgstab \cong A\grstab/\Omega(1)$.  This means that all computations in $A\dgstab$ can be reduced to computations in $A\grstab$:  For each indecomposable object $X$ in $A\dgstab$, there is a unique power of $\Omega(1)$ such that $\Omega^nX(n) \cong X$ is concentrated in degrees $[-d, 0]$.  If $X$ and $Y$ both lie in this degree range, then $\Hom_{A\dgstab}(X,Y)$ and $\Hom_{A\grstab}(X,Y)$ coincide.  The projection functor $A\grstab \rightarrow A\dgstab$ is exact, and every distinguished triangle in $A\dgstab$ descends from one in $A\grstab$.  Since $\Omega^{-1} \cong (1)$, we can also view the grading shift functor as the suspension functor.

Every indecomposable object in $A\dgstab$ is isomorphic to a shift of $M^1_l$, for $1 \le l \le \lfloor \frac{n+1}{2} \rfloor$.  We say such an object has \textbf{length} $l$.  In $A\dgstab$, the grading shift functor is periodic:  $X(P) \cong X$ for $P = (n+1)(d+2) - 2$ and any $X \in A\dgstab$.  Additionally, if $X$ has length exactly $\frac{n+1}{2}$, then $X(\frac{P}{2}) \cong X$.

It is proven in \cite{brightbill2018differential} that $A\dgstab$ is Hom-finite and Krull-Schmidt.  By Proposition \ref{Calabi-Yau}, $A\dgstab$ is $-(d+1)$-Calabi-Yau.  As in Definition \ref{simple-minded tuple}, let $\widehat{\mathcal{E}}$ be the set of all orthogonal $n$-tuples of objects of $A\dgstab$ (up to isomorphism).  Let $\mathcal{E}$ be the subset of all generating tuples.

Our primary goal is the analysis of the action of $\Xi$ on $\mathcal{E}$.  We will show that $A\dgstab$ admits $(d+1)$-orthogonal maximal extensions in Theorem \ref{well-defined}.  We shall see in Corollary \ref{object transitivity} that $\mathcal{E} = \widehat{\mathcal{E}}$ and that the action of $\Xi$ is transitive.  Note that the simple modules $(S_1, \cdots, S_n)$ form a generating tuple, hence $\mathcal{E}$ is non-empty.


\subsection{Beads on a Wire}
We now develop a combinatorial model of $\mathcal{E}$.  We shall associate indecomposable objects of $A\dgstab$ to beads of varying lengths on a circular wire.

We consider the set $\Z/P\Z$, viewed as a collection of evenly-spaced points on a circular wire of length $P$.  For integers $i, j$, we shall denote by $[[i, j]]$ the image of the closed interval $[i, j]\cap \Z$ in $\Z / P \Z$.
\begin{definition}
Let $i, l$ be integers, with $1 \le l \le n$.  Define $B_l(i)$ to be the interval $[[i - l(d+2), i]]$.  We refer to $B_l(i)$ as a \textbf{bead of type $l$ in position $i$}.  We refer to the interval $[[i - l(d+2) + 1, i - 1]]$ as the \textbf{well} of $B_l(i)$.  The intervals $[[i - l(d+2), i- l(d+2) + 1]]$ and $[[i-1, i]]$ are the \textbf{ridges} of $B_l(i)$, and the points $i - l(d+2)$ and $i$ are the \textbf{endpoints} of $B_l(i)$.
\end{definition}

\begin{remark}
We shall often identify the integer $i$ in the above definition with its image in $\Z / P\Z$.  This shall cause no confusion, as the definition depends only on the image of $i$.  We shall also view $(j)$ as a shift operator on the set of beads, so that $B_l(i)(j) = B_l(i+j)$.
\end{remark}

The total length of a bead of type $l$ is $l(d+2)$.  Geometrically, we view the beads as possessing an interior well, a depression of length $l(d+2) -2$ into which other (smaller) beads may be placed.  This well is surrounded by two ridges of length one, over which other beads cannot be placed.  We give an illustration in Figure \ref{beads1}.  Since no beads can fit in the well of a bead of type $1$, we will depict these beads without ridges or a well; this is a purely aesthetic choice.

\begin{definition}
\label{overlap}
Let $1 \le r \le l \le n$, and let $i, j \in \Z$.  We say the beads $B_l(i)$ and $B_r(j)$ \textbf{do not overlap} if one of the following holds:\\
1)  $[[j - r(d+2), j]] \subset [[i - l(d+2) + 1, i - 1]]$; that is, $B_r(j)$ is contained within the well of $B_l(i)$.\\
2)  $[[j - r(d+2), j]] \subset [[i, i - l(d+2) + P]]$; that is, $B_r(j)$ lies outside of $B_l(i)$ (though the beads' endpoints may touch).
\end{definition}

\begin{remark}  Note that condition 2) is symmetric with respect to $B_l(i)$ and $B_r(j)$, and condition 1) can only occur if $r < l$.
\end{remark}

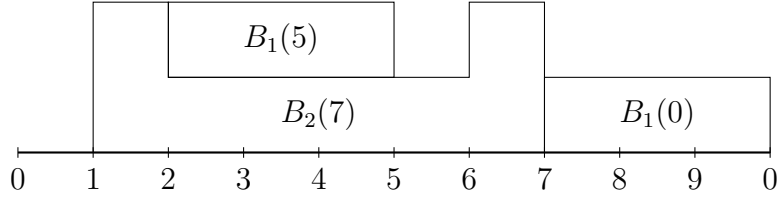
\begin{figure}
\centering
\begin{tikzpicture}
\draw[thick] (0,0) -- (10, 0) ;
\foreach \x in {0, 1, 2, 3, 4, 5, 6, 7, 8, 9}
	\draw (\x cm, 2pt) -- (\x cm, -2pt) node[anchor=north] {$\x$};
\draw (10 cm, 2pt) -- (10 cm, -2pt) node[anchor=north] {$0$};

\draw (1, 0) -- (7,0) -- (7, 2) -- (6,2) -- (6,1) -- (2,1) -- (2,2) --(1,2) -- cycle;
\draw (4, 0.15) node[anchor = south] {$B_2(7)$};

\draw (2, 1) rectangle (5, 2);
\draw (3.5, 1.15) node[anchor = south] {$B_1(5)$};

\draw (7, 0) rectangle (10, 1);
\draw (8.5, 0.15) node[anchor = south] {$B_1(0)$};
\end{tikzpicture}

\caption{Three non-overlapping beads; $n = 3$, $d=1$}
\label{beads1}
\end{figure}

In Figure \ref{beads1}, none of the beads overlap with one another.  Bead $B_1(5)$ lies inside the well of $B_2(7)$.  Bead $B_1(0)$ lies outside of $B_2(7)$.

Since the length of the wire is $P = (n+1)(d+2) -2$, there is not quite enough room on the wire for two beads of types $l$ and $n+1-l$.  However, two such beads can be placed on the wire in such a way that they intersect precisely along their ridges.  This motivates the following definition and proposition:

\begin{definition}
Let $B_l(i)$ be a bead.  Define the \textbf{partner} of $B_l(i)$ to be the bead $\widetilde{B_l(i)} := B_{n+1-l}(i-l(d+2) +  1)$.
\end{definition}

It is easy to verify that $B_l(i)$ and $\widetilde{B_l(i)}$ intersect precisely along the ridges of both beads, and that the function taking a bead to its partner is an involution.  See Figure \ref{beads2}.

If $B_l(i)$ and $B_r(j)$ are two beads, note that $B_r(j)$ lies in the well of $B_l(i)$ if and only if $B_r(j)$ lies outside $\widetilde{B_l(i)}$.  In this case, then $\widetilde{B_r(j)}$ and $B_l(i)$ will necessarily overlap, and $\widetilde{B_l(i)}$ will lie in the well of $\widetilde{B_r(j)}$.

We now relate beads to the indecomposable objects of $A\dgstab$.  

\begin{definition}
\label{bead object mapping}
Given a bead $B_l(i)$, define the \textbf{associated object of} $B_l(i)$ to be the object $M^1_l(i) \in A\dgstab$.  Let $\Phi$ be denote the function mapping a bead to (the isomorphism class of) its associated object.
\end{definition}

\begin{figure}
\centering
\begin{tikzpicture}
\draw[thick] (0,0) -- (10, 0) ;
\foreach \x in {0, 1, 2, 3, 4, 5, 6, 7, 8, 9}
	\draw (\x cm, 2pt) -- (\x cm, -2pt) node[anchor=north] {$\x$};
\draw (10 cm, 2pt) -- (10 cm, -2pt) node[anchor=north] {$0$};

\draw (1, 0) -- (7,0) -- (7, 2) -- (6,2) -- (6,1) -- (2,1) -- (2,2) --(1,2) -- cycle;
\draw (4, 0.15) node[anchor = south] {$B_2(7)$};
\draw[red, dashed, very thick] (0,0) -- (2, 0) -- (2, 2) -- (1, 2) -- (1, 1) -- (0, 1);
\draw[red, dashed, very thick] (10, 0) -- (6, 0) -- (6, 2) -- (7, 2) -- (7, 1) -- (10, 1);
\draw (9, 0.15) node[anchor = south] {$\widetilde{B_2(7)}$};
\end{tikzpicture}

\caption{A bead and its partner; $n = 3$, $d=1$}
\label{beads2}
\end{figure}
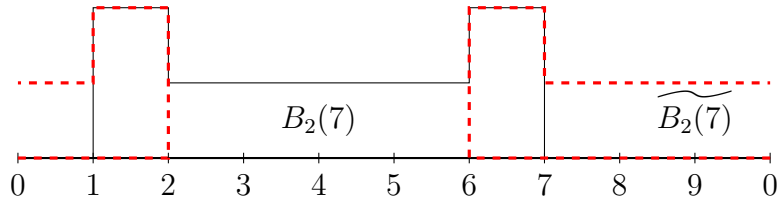

\begin{remark}
Note that $B_l(i+P) = B_l(i)$ for any $i$, so that $\Phi$ is well-defined when viewed as a function of $l$ and $i$.
\end{remark}

\begin{proposition}
\label{two-to-one}
$\Phi$ defines a two-to-one map from the set of beads onto the set of (isomorphism classes of) indecomposable objects of $A\dgstab$.  Each bead has the same image as its partner.
\end{proposition}
\begin{proof}
By Proposition 6.10 of \cite{brightbill2018differential}, every object of $A\dgstab$ is isomorphic to $M^1_l(i) = \Phi(B_l(i))$ for some $1 \le l \le \lfloor \frac{n+1}{2} \rfloor$, $0 \le i < P$; thus $\Phi$ is surjective.  A straightforward counting argument shows that there are $nP$ beads, and Corollary 6.15 of \cite{brightbill2018differential} shows that $A\dgstab$ has $\frac{nP}{2}$ indecomposable objects, up to isomorphism.

A straightforward calculation using Proposition 6.9, Equation (8) of \cite{brightbill2018differential} shows that $\Phi(\widetilde{M^1_l(i)}) \cong \Omega \Phi(M^1_l(i))(1) \cong \Phi(M^1_l(i))$.  Since every indecomposable object has at least two preimages under $\Phi$, it follows by the pigeonhole principal that $\Phi$ is two-to-one.
\end{proof}

\begin{remark}
Note that if $l < \frac{n+1}{2}$, each $M^1_l(i)$ is the associated object of a unique bead of type $l$ and a unique bead of type $n+1 -l$.  When $l = \frac{n+1}{2}$ (note this requires $n$ to be odd), both preimages of $M^1_l(i)$ are beads of type $\frac{n+1}{2}$.  Taking the partner of a bead corresponds to applying $\Omega(1)$ to its associated object.
\end{remark}

\begin{proposition}
\label{bead-module compatibility}
Let $1 \le r \le l \le \lfloor \frac{n+1}{2} \rfloor$.  Then the beads $B_l(i)$ and $B_r(j)$ do not overlap if and only if $\Phi(B_l(i))$ and $\Phi(B_r(j))$ are distinct and independent.
\end{proposition}

To prove the above Proposition, it will be helpful to reformulate Definition \ref{overlap}.
\begin{lemma}
\label{overlap reformulation}
Let $1 \le r \le l \le n$.  Two beads $B_l(i)$ and $B_r(j)$ do not overlap if and only if $[[j - r(d+2) +1, j]] \cap \{i - l(d+2) + 1, i\} = \emptyset$ (as subsets of $\mathbb{Z}/P\mathbb{Z}$).
\end{lemma}
\begin{proof}
Note that changing $i$ and $j$ by a multiple of $P$ does not affect the statement of the lemma.

Suppose $B_l(i)$ and $B_r(j)$ do not overlap.  Suppose condition 1) of Definition \ref{overlap} holds, that is, $B_r(j)$ lies in the well of $B_l(i)$.  Then $i$ and $j$ may be chosen so that
\begin{equation*}
i - l(d+2) + 1 <  j - r(d+2) +1 < j < i < i + l(d+2) + 1 + P
\end{equation*}
Thus neither $i$ nor $i - l(d+2)+1$ lies in $[[j - r(d+2) + 1, j]]$, hence the intersection is empty.

If condition 2) holds (i.e., $B_r(j)$ lies outside of $B_l(i)$), then $i$ and $j$ may be chosen so that
\begin{equation*}
i < j-r(d+2) +1 < j \le i-l(d+2) + P < i - l(d+2) + 1 + P < i +P
\end{equation*}
Once again, the intersection is empty.  This proves the forward direction.

For the reverse direction, suppose the intersection is empty.  Then, of the six potential cyclic orderings of $\{j-r(d+2)+1, j, i-l(d+2)+1, i\}$ inside $\Z/P\Z$, the only two consistent possibilities are:
\begin{equation*}
i-l(d+2)+1 < j-r(d+2) + 1 < j < i < i-l(d+2)+1+P
\end{equation*}
or
\begin{equation*}
j -r(d+2) +1 < j < i - l(d+2)+1 < i < j-r(d+2)+1+P
\end{equation*}
The first case implies that $B_r(j)$ lies in the well of $B_l(i)$.  The second implies that $B_r(j)$ lies outside of $B_l(i)$.  In both cases, $B_l(i)$ and $B_r(j)$ do not overlap.
\end{proof}

We now prove Proposition \ref{bead-module compatibility}.

\begin{proof}
$\Phi(B_l(i))$ and $\Phi(B_r(j))$ are distinct and independent in $A\dgstab$ if and only if, for all $0 \le m \le d$,
\begin{equation*}
\Hom(M^1_l(i), M^1_r(j-m)) = \Hom(M^1_r(j), M^1_l(i-m)) = 0
\end{equation*}

Since $A\dgstab$ is $-(d+1)$-Calabi-Yau, we can rewrite the above condition as
\begin{equation*}
\Hom(M^1_l, M^1_r(j-i-m)) = \Hom(M^1_l, M^1_r(j-i+m-d-1)) = 0
\end{equation*}
for all $0 \le m \le d$.  This can be further simplified to
\begin{equation*}
\Hom(M^1_l, M^1_r(j-i-m)) = 0
\end{equation*}
for all $0 \le m \le d+1$.

By Theorem 6.12 of \cite{brightbill2018differential} this holds if and only if, for all $0 \le m \le d+1$,
\begin{align*}
&j-i-m & \notin & \{(d+2)(r-k) \mid 1 \le k \le r\} \cup\\
& & &  \{(d+2)(n+1-k)-1 \mid 1+l-r \le k \le l\}\\
&& \Updownarrow&\\
&j-i & \notin & [[0, (d+2)(r-1) + d+1]] \cup\\
& & & [[(d+2)(n+1-l) -1, (d+2)(n+r-l) +d]]\\
& & \Updownarrow &\\
&j-i &\notin & [[0, (d+2)r -1]] \cup\\
& & & [[-(d+2)l +1, (d+2)(r-l)]]\\
& & \Updownarrow &\\
&j-i, &\notin & [[0, (d+2)r -1]]\\
&j-i+l(d+2)-1 &  &\\
& & \Updownarrow &\\
&i, i-l(d+2)+1 &\notin & [[j- r(d+2) +1, j]]
\end{align*}
where the above sets are viewed as subsets of $\Z/P\Z$ if $r < \frac{n+1}{2}$ and as subsets of $\Z/(\frac{P}{2}) \Z$ if $r = \frac{n+1}{2}$.

If $r < \frac{n+1}{2}$, the last condition is precisely that which appears in Lemma \ref{overlap reformulation}, and we are done.

If $r = \frac{n+1}{2}$, then necessarily $l = \frac{n+1}{2}$.  In this case, $B_r(j)$ and $B_l(i)$ always overlap, since the length of both the well and the outside of $B_l(i)$ is $\frac{n+1}{2}(d+2) - 2$, which less than the length of $B_r(j)$.  Thus it suffices to show that there are no pairs of distinct, independent objects of length $\frac{n+1}{2}$.  Since $-r(d+2) + 1 = -\frac{P}{2}$, the interval $[[j - r(d+2)+1, j]]$ is equal to $\Z/ (\frac{P}{2})\Z$, hence there can be no pairs of distinct, independent objects of length $\frac{n+1}{2}$.
\end{proof}

Proposition \ref{bead-module compatibility} can be partially extended to beads of unrestricted length.

\begin{proposition}
\label{bead-module compatibility 2}
Let $1 \le r, l \le n$.  If $B_l(i)$ and $B_l(j)$ do not overlap, then $\Phi(B_l(i))$ and $\Phi(B_r(j))$ are distinct and independent.
\end{proposition}

\begin{proof}
Without loss of generality, we may assume $r \le l$.  By Proposition \ref{two-to-one}, a bead and its partner have the same associated object, so we can replace any bead with its partner without affecting the conclusion of the Proposition.  We shall reduce to the case where $r \le l \le \lfloor \frac{n+1}{2} \rfloor$ and apply Proposition \ref{bead-module compatibility}.

We may assume that $l >  \lfloor \frac{n+1}{2} \rfloor$.  Since $r \le l$, note that either $B_r(j)$ is contained in the well of $B_l(i)$ or lies outside.  In either case $\widetilde{B_l(i)}$ and $B_r(j)$ do not overlap, and $\widetilde{B_l(i)}$ has type $n+1 - l < \lfloor \frac{n+1}{2} \rfloor$.  Thus if $r \le \lfloor \frac{n+1}{2} \rfloor$, we have completed the reduction.  Otherwise, $r > \lfloor \frac{n+1}{2} \rfloor$, hence $B_r(j)$ is longer than $\widetilde{B_l(i)}$.  Repeating the same argument as above, $\widetilde{B_r(j)}$ and $\widetilde{B_l(i)}$ do not overlap, and both beads have type less than $\lfloor \frac{n+1}{2} \rfloor$, hence their associated objects are distinct and independent.  
\end{proof}

\begin{remark}
The converse to Proposition \ref{bead-module compatibility 2} is false.  Given a pair of non-overlapping beads, by replacing beads with their partners we can obtain four distinct pairs of beads with the same image under $\Phi$.  Of these four pairs, exactly one will overlap.
\end{remark}


\subsection{Bead Arrangements}

We now translate the notion of an orthogonal tuple into the language of beads.

\begin{definition}
A \textbf{colored bead arrangement} is an $n$-tuple whose entries are mutually non-overlapping beads.  An (uncolored) \textbf{bead arrangement} is a set of $n$ mutually non-overlapping beads.  A \textbf{free bead arrangement} is a bead arrangement, taken up to a rotation of the wire.  We let $CBA$ (resp. $BA$, $FBA$) denote the set of all colored bead arrangements (resp. bead arrangements, free bead arrangements).
\end{definition}

\begin{definition}
Let $B$ be a bead in a (colored, uncolored, or free) bead arrangment $A$.  Define the \textbf{height} $H(B)$ of $B$ in $A$ to be the number of beads $B'$ in $A$ such that $B \subset B'$.  If $B$ is in a colored bead arrangement, we define the \textbf{color} of $B$ to be the integer $i \in \{1, \cdots, n\}$ such that $B$ is the $i$th entry of the tuple.
\end{definition}

Figure \ref{beads1} shows an uncolored free bead arrangement.  The height of $B_1(5)$ is $2$, and the height of the other two beads is $1$.

The following statement is an immediate corollary of Proposition \ref{bead-module compatibility}.
\begin{proposition}
\label{arrangement-generator compatibility 1}
$\Phi$ induces surjections
\begin{eqnarray*}
CBA \twoheadrightarrow \widehat{\mathcal{E}}\\
FBA \twoheadrightarrow \widehat{\mathcal{E}}/\sim
\end{eqnarray*}
\end{proposition}
\begin{proof}
Choose $(X_i)_i \in \widehat{\mathcal{E}}$.  By Proposition \ref{two-to-one}, for each $i$ there exists a bead $B^i$ of type $l$, $1 \le l \le \lfloor \frac{n+1}{2} \rfloor$, such that $\Phi(B^i) = X_i$.  By Proposition \ref{bead-module compatibility}, the $B^i$ are mutually non-overlapping, since the $X_i$ are mutually independent.  Thus $(B^i)_i$ is a colored bead arrangement and $\Phi(B^i)_i= (X_i)_i$.  If $C^i$ is a bead obtained from $B^i$ by a rotation of the wire, then $\Phi(C^i)$ is a shift of $\Phi(B^i)$.  Thus the second function is well-defined, and it is clear from the above argument that it is surjective.
\end{proof}

We would like to further restrict the class of bead arrangements so that the surjective maps defined above become bijections.  The above proof suggests that we restrict our attention to arrangements in which only beads of type $1 \le l \le \lfloor \frac{n+1}{2} \rfloor$ are permitted.  For $n$ even, this is the correct solution, as $\Phi$ induces a bijection between the beads of type $1 \le l < \frac{n+1}{2}$ and isomorphism classes of indecomposable objects of length $l$.  However, if $n$ is odd, and $l = \frac{n+1}{2}$, then $\Phi$ is two-to-one on the set of beads of type $l$.  To resolve this issue, we introduce a new object to our combinatorial model.

\begin{definition}
A \textbf{circlet} is a set $C(i) = \{B_{\frac{n+1}{2}}(i), \widetilde{B_{\frac{n+1}{2}}(i)}\}$ consisting of a bead of type $\frac{n+1}{2}$ and its partner.
\end{definition}

Geometrically, we interpret $C(i)$ as both beads, glued along their overlapping boundaries.  (See Figure \ref{beads3}.)  Thus, $C(i)$ divides the ring into two wells of length $\frac{n+1}{2}(d+2) -2$, separated by the two ridges $[[i-1, i]]$ and $[[i-\frac{n+1}{2}(d+2), i-\frac{n+1}{2}(d+2) + 1]]$.  Note that we can apply $\Phi$ to $C(i)$, since both elements of $C(i)$ have the same image under $\Phi$.  Furthermore, $\Phi$ establishes a bijection between the set of circlets and the set of isomorphism classes of indecomposable objects of length $\frac{n+1}{2}$.  

\begin{definition}
A bead $B_r(j)$ and a circlet $C(i)$ \textbf{do not overlap} if $B_r(j)$ does not overlap with either bead in $C(i)$.
\end{definition}

Note that if $r \ge \frac{n+1}{2}$, $B_r(j)$ will always overlap with at least one of the beads in any circlet $C(i)$, and if $r < \frac{n+1}{2}$, then if $B_r(j)$ does not overlap with one of the beads in $C(i)$, it will not overlap with either.

\begin{definition}
A (colored, uncolored, or free) bead arrangement is called \textbf{reduced} if all beads in the arrangement have type $1 \le l \le \lfloor \frac{n+1}{2} \rfloor$, and any bead $B_{\frac{n+1}{2}}(i)$ is replaced by the corresponding circlet $C(i)$.  We denote by $RCBA$ (resp. $RBA$, $RFBA$) the set of reduced colored bead arrangements (resp. reduced bead arrangements, reduced free bead arrangements).
\end{definition}

Note that since any two beads of type $\frac{n+1}{2}$ overlap, there can be at most one circlet in any type of reduced bead arrangement.  A reduced colored bead arrangement with a circlet is shown in Figure \ref{beads3}.

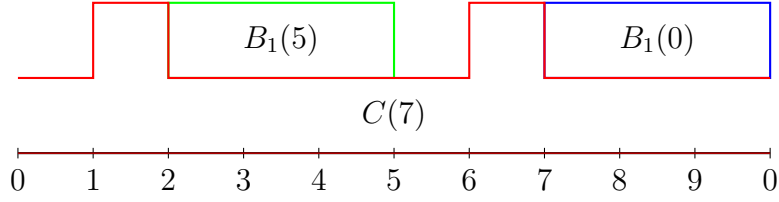
\begin{figure}
\centering
\captionsetup{justification = centering}
\begin{tikzpicture}
\draw[thick] (0,0) -- (10, 0) ;
\foreach \x in {0, 1, 2, 3, 4, 5, 6, 7, 8, 9}
	\draw (\x cm, 2pt) -- (\x cm, -2pt) node[anchor=north] {$\x$};
\draw (10 cm, 2pt) -- (10 cm, -2pt) node[anchor=north] {$0$};
\draw[green, thick] (2, 1) rectangle (5, 2);
\draw (3.5, 1.15) node[anchor = south] {$B_1(5)$};

\draw[blue, thick] (7, 1) rectangle (10, 2);
\draw (8.5, 1.15) node[anchor = south] {$B_1(0)$};

\draw[red, thick] (10, 1) -- (7,1) -- (7, 2) -- (6,2) -- (6,1) -- (2,1) -- (2,2) --(1,2) -- (1, 1) -- (0, 1);
\draw[red] (0, 0) -- (10, 0);
\draw (5, 0.15) node[anchor = south] {$C(7)$};
\end{tikzpicture}
\caption{The reduced colored bead arrangement $(C(7) , B_1(5), B_1(0))$; $n = 3$, $d=1$}
\label{beads3}
\end{figure}

\begin{proposition}
\label{arrangement-generator compatibility 2}
$\Phi$ induces bijections
\begin{eqnarray*}
RCBA \leftrightarrow \widehat{\mathcal{E}}\\
RFBA \leftrightarrow \widehat{\mathcal{E}}/\sim
\end{eqnarray*}
\end{proposition}
\begin{proof}
Surjectivity of both maps follows immediately from the proof of Proposition \ref{arrangement-generator compatibility 1}.  Since $\Phi$ induces a bijection between beads of type $1 \le l < \frac{n+1}{2}$ and isomorphism classes of indecomposable objects of length $l$, as well as between circlets and isomorphism classes of indecomposable objects of length $\frac{n+1}{2}$, it follows that both maps are injective.
\end{proof}


\subsection{Counting $\widehat{\mathcal{E}}$}

In this section, we determine the cardinality of $\widehat{\mathcal{E}}$.  By Proposition \ref{arrangement-generator compatibility 2}, it suffices to count the number of reduced colored bead arrangements.  It is easy to reduce the problem to counting the reduced free bead arrangements.

\begin{proposition}
\label{colored to free reduction}
$|RCBA| = n!\cdot P \cdot |RFBA|$
\end{proposition}
\begin{proof}
The canonical map $RCBA \rightarrow RBA$ sending an $n$-tuple to a set is clearly surjective and $n!$-to-one.  The canonical map $RBA \rightarrow RFBA$ sending a bead arrangement to its equivalence class under rotation is clearly surjective and $P$-to-one.
\end{proof} 

Given a bead arrangement, one can draw a plane tree by associating a vertex to each bead and drawing an edge to each bead sitting directly on top of it.

\begin{definition}
\label{bead tree}
Given an uncolored bead arrangement $A$, define the rooted plane tree $(P(A), r_A)$ as follows:  The vertices of $P(A)$ are the beads of $A$, plus a new vertex, $r_A$, associated to the wire.  The vertex $r_A$ is defined to be the root of the tree.  Draw edges between $r_A$ and each bead of height $1$.  Draw an edge between $B_l(i)$ and $B_r(j)$ if and only if the difference in height between the two beads is $1$ and one bead contains the other.  Associating the vertex $B_l(i)$ with $i \in \mathbb{Z}/P\mathbb{Z}$, the natural cyclic ordering on $\mathbb{Z}/P\mathbb{Z}$ induces a cyclic ordering of all non-root vertices.  This induces a cyclic ordering of the edges around each vertex of height $\neq 1$.  For a vertex $B_l(i)$ of height one, the edge incident to $r_A$ is ordered as though it had value $i$.

We refer to $(P(A), r_A)$ as the \textbf{tree associated to} $A$.  The isomorphism class of $(P(A), r_A)$ (as a rooted plane tree) is called the \textbf{class} of $A$.
\end{definition}

We give two examples of bead arrangements and their associated trees in Figure \ref{beads4}.  The root of each tree is the bottom-most vertex.  Note that the two trees in Figure \ref{beads4} are isomorphic as trees, but not as plane trees, hence the two arrangements do not have the same class.  Intuitively, two bead arrangements will have the same class if and only if they differ by a rigid motion, where beads are allowed to move along, but not through, each other.  In particular, it is straightforward to check that $P(A)$ is invariant under rotation of the wire, hence the map $A \mapsto (P(A), r_A)$ is defined for free bead arrangements.

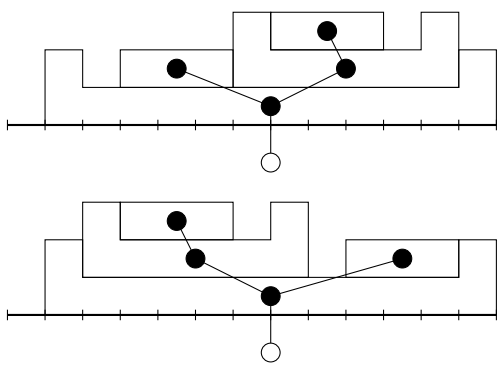
\begin{figure}[H]
\centering
\begin{tikzpicture}[x = 0.5 cm, y = 0.5 cm]
\draw[thick] (0,0) -- (13, 0) ;
\foreach \x in {0, 1, 2, 3, 4, 5, 6, 7, 8, 9, 10, 11, 12, 13}
	\draw (\x, 2pt) -- (\x, -2pt);

\draw (1, 0) -- (13, 0) -- (13, 2) -- (12, 2) -- (12,1) -- (2,1) -- (2,2) --(1,2) -- cycle;
\draw (3, 1) rectangle (6, 2);
\draw (6, 1) -- (12, 1) -- (12, 3) -- (11, 3) -- (11, 2) -- (7, 2) -- (7, 3) -- (6, 3) -- cycle;
\draw (7, 2) rectangle (10, 3);
\draw (7, -1) circle (0.25);
\filldraw[fill = black] (7, 0.5) circle (0.25);
\filldraw[fill = black] (4.5, 1.5) circle (0.25);
\filldraw[fill = black] (9, 1.5) circle (0.25);
\filldraw[fill = black] (8.5, 2.5) circle (0.25);
\draw (7, -0.75) -- (7, 0.5);
\draw (7, 0.5) -- (4.5, 1.5);
\draw (7, 0.5) -- (9, 1.5);
\draw (9, 1.5) -- (8.5, 2.5);
\end{tikzpicture}

\vspace{10 pt}

\begin{tikzpicture}[x = 0.5 cm, y = 0.5 cm]
\draw[thick] (0,0) -- (13, 0) ;
\foreach \x in {0, 1, 2, 3, 4, 5, 6, 7, 8, 9, 10, 11, 12, 13}
	\draw (\x, 2pt) -- (\x, -2pt);
\draw (1, 0) -- (13, 0) -- (13, 2) -- (12, 2) -- (12,1) -- (2,1) -- (2,2) --(1,2) -- cycle;
\draw (9, 1) rectangle (12, 2);
\draw (2, 1) -- (8, 1) -- (8, 3) -- (7, 3) -- (7, 2) -- (3, 2) -- (3, 3) -- (2, 3) -- cycle;
\draw (3, 2) rectangle (6, 3);
\draw (7, -1) circle (0.25);
\filldraw[fill = black] (7, 0.5) circle (0.25);
\filldraw[fill = black] (10.5, 1.5) circle (0.25);
\filldraw[fill = black] (5, 1.5) circle (0.25);
\filldraw[fill = black] (4.5, 2.5) circle (0.25);
\draw (7, -0.75) -- (7, 0.5);
\draw (7, 0.5) -- (10.5, 1.5);
\draw (7, 0.5) -- (5, 1.5);
\draw (5, 1.5) -- (4.5, 2.5);
\end{tikzpicture}

\caption{Two bead arrangements and their associated plane trees; $n = 4$, $d=1$}
\label{beads4}
\end{figure}

The following properties of the map $A \mapsto (P(A), r_A)$ are straightforward to verify.  We refer to Section \ref{Rooted Plane Trees} for terminology.

\begin{proposition}
\label{bead tree properties}
Let $A$ be a (free or uncolored) bead arrangement, and let $B_l(i)$ be a bead in $A$.  Then:\\
1)  The depth of the vertex $B_l(i)$ in $P(A)$ is equal to the height of $B_l(i)$ in $A$.\\
2)  The weight of $B_l(i)$ in $(P(A), r_A)$ is $l$.\\
3)  If $B_l(i)$ has height one, let $A'$ denote the bead arrangement obtained by replacing $B_l(i)$ with its partner.  Then there is a canonical isomorphism of plane trees $P(A) \xrightarrow{\sim} P(A')$ induced by identifying the trees' common non-root vertices.  Identifying the two trees via this isomorphism, $(P(A'), r_{A'})$ is a rebalancing of $(P(A), r_A)$ in the direction of $B_l(i)$.
\end{proposition}

\begin{proof}
Given two beads $B$ and $B'$ in $A$, $B$ is an ancestor of $B'$ in $(P(A), r_A)$ if and only if $B' \subsetneq B$.  The first statement follows.

For the second statement, we first prove that $W(B_l(i)) \le l$, by induction on $l$.  If $l = 1$, then the statement is immediate, since $B_1(i)$ contains no bead but itself and is therefore a leaf.  Suppose the result holds for beads of type $k < l$.  Suppose the children of $B_l(i)$ are $\{ B_{x_j}(y_j) \}_{j=1}^{s}$.  Since a bead of type $k$ has length $k(d+2)$, and since the beads $B_{x_j}(y_j)$ are mutually non-overlapping beads inside the well of $B_l(i)$, we have that $\sum_{j} x_j \le l -1$.  Thus,
\begin{align}
\label{btp1}
W(B_l(i)) = 1+ \sum_j W(B_{x_j}(y_j)) \le 1 + \sum_j x_j \le l
\end{align}
and the inductive step is complete.

Note that (\ref{btp1}) remains true if $B_l(i)$ is replaced by $r_A$, and $l$ by $n+1$, since the ring has the same length as the well of a (hypothetical) bead of type $n+1$.  Furthermore, if equality holds in (\ref{btp1}), then $W(B_{x_j}) = x_j$ for all $j$.  Thus, equality at a vertex $v$ implies equality at all descendants of $v$.  Equality holds at $r_A$ by construction, hence at all vertices.  This proves the second statement.

The isomorphism in the third statement identifies $r_A$ with $\widetilde{B_l(i)}$ and $B_l(i)$ with $r_{A'}$; the remaining vertices are shared by the two trees.  The rest of the statement follows directly from definitions.
\end{proof}

Motivated by the previous proposition, if $A$ is a (free or uncolored) bead arrangement, and $A'$ is a bead arrangement obtained from $A$ by replacing a height one bead $B_l(i)$ by its partner, we say that $A'$ is a \textbf{rebalancing of $A$ in the direction of $B_l(i)$}.  Thus, the third statement of the previous proposition can be restated as saying that the rebalancing operation commutes with taking the associated tree of a bead arrangement.  By repeatedly rebalancing a tree in the direction of vertices of weight greater than $\lfloor \frac{n+1}{2} \rfloor$, one eventually obtains a balanced tree.  Performing the corresponding operation on bead arrangements, we see that reduced bead arrangements are precisely the analogues of balanced trees.  More precisely:

\begin{corollary}
\label{bead arrangement rebalancing}



Let $A$ be a free bead arrangement.  Then $A$ defines a reduced free bead arrangement $\overline{A}$ if and only if $(P(A), r_A)$ is balanced.  In this case, $\overline{A}$ contains a circlet $C = \{B, \widetilde{B}\}$ if and only if $P(A)$ has two balancing roots.  In this case, let $A'$ be the other free bead arrangement defining $\overline{A}$, with $B \in A$ and $\widetilde{B} \in A'$.  After identifying $P(A)$ and $P(A')$ via the canonical isomorphism, $B$ and $\widetilde{B}$ are the two balancing roots of the tree.
\end{corollary}

\begin{proof}
$A$ defines a reduced free bead arrangement if and only if every bead is of type at most $\frac{n+1}{2}$.  By Proposition \ref{bead tree properties}, this holds if and only if $P(A)$ is balanced.  If $A$ defines a reduced free bead arrangement, $\overline{A}$ contains a circlet if and only if $A$ contains a height one bead of type $\frac{n+1}{2}$, if and only if $P(A)$ contains a depth one vertex of weight $\frac{n+1}{2}$.  By Proposition \ref{tree balancing}, this happens if and only if $P(A)$ has two balancing roots.  In this case, $A'$ is a rebalancing of $A$ in the direction of $B$ and $A$ is a rebalancing of $A'$ in the direction of $\widetilde{B}$.  Identifying the two trees, $B$ and $B'$ are the vertices incident to the edge of weight $\frac{n+1}{2}$ and thus are the two balancing roots.
\end{proof}

By Corollary \ref{bead arrangement rebalancing}, the map $\overline{A} \mapsto P(\overline{A})$ is well-defined for reduced free bead arrangements, if we interpret $P(\overline{A})$ as an isomorphism class of plane trees.  We define $P(\overline{A})$ to be the \textbf{class} of $\overline{A}$, as for free bead arrangements.

We are now ready to count the reduced free bead arrangements.  The key result is the following lemma:

\begin{lemma}
\label{single tree counting}
Let $(T, r)$ be a rooted plane tree with $n$ edges.  Then the number of free bead arrangements $A$ of class $(T, r)$ is
\begin{align}
\label{stc1}
N_{T, r} = \binom{d + |c(r)| - 1}{d}\prod_{v \in V_T - \{r\}} \binom{d + |c(v)|}{d}
\end{align}
\end{lemma}

\begin{proof}
We describe a choice procedure for constructing an arbitrary free bead arrangement of class $(T, r)$.  Starting with the root and working upwards, we associate beads to the children of each vertex of the tree.

First, we specify the placement of the height one beads.  Write $c(r) = \{v_1 < v_2 < \cdots < v_k < v_1\}$ as a cyclically ordered set.  (Note $k = |c(r)|$.)  We shall place beads of type $W(v_1), W(v_2), \cdots, W(v_k)$ sequentially on the wire, so that their right edges form an increasing sequence in $\mathbb{Z}/P\mathbb{Z}$.  Since a free bead arrangement is defined up to a rotation of the wire, we assume without loss of generality that the right edge of the bead corresponding to $v_1$ is at position $0$.  Since the $i$th bead has type $W(v_i)$, and $\sum^k_{i=1} W(v_i) = n$, the $k$ beads take up a total of $n(d+2)$ space on the wire, which has total length $n(d+2) + d$.  Thus, to uniquely specify the position of the height one beads (up to rotation of the wire), we need to distribute the $d$ empty spaces amongst the $k$ gaps between beads.  There are $\binom{d + k - 1}{d} = \binom{d + |c(r)| -1}{d}$ such choices.

Next, given any vertex $v$ corresponding to a bead $B$ of type $l$ already placed by our choice procedure, we must place the beads which lie in the well of $B$ and have height $H(B) + 1$.  It is clear that the number of such placements depends only on the type of $B$ and is independent of its horizontal placement.  Thus we may identify the well of $B$ with the interval $[0, l(d+2) -2 ]$.  Since $v \neq r$, $C(v) = \{w_1< \cdots < w_k \}$ is totally ordered (i.e., the parent of $v$ lies between $w_k$ and $w_1$ in the cyclic ordering).  As before, we place beads of type $W(w_1), \cdots, W(w_k)$ sequentially, from left to right.  Once again, there is a total of $d$ empty space in the well of $B$, and uniquely specifying the position of the beads in the well is equivalent to distributing $d$ empty spaces amongst the $k+1$ gaps found between the $k$ beads and the two walls of $B$.  There are $\binom{d+k}{d} = \binom{d + |c(v)|}{d}$ such choices.

It is clear that this choice procedure uniquely specifies all free bead arrangements of class $(T, r)$.  Since the choices made at each vertex are independent of previous choices, this establishes the formula.
\end{proof}

\begin{corollary}
\label{independence of root}
Let $(T,r)$ be a rooted plane tree.  The quantity $N_{T, r}$ is independent of $r$.
\end{corollary}

\begin{proof}
It suffices to show that $N_{T,r} = N_{T,r'}$ for adjacent vertices $r$ and $r'$.  Let $X_r$ denote the set of free bead arrangements of class $(T, r)$, and similarly for $r'$.  For each arrangement $A \in X_r$, there is a unique bead $B$ of height $1$ corresponding to the vertex $r'$; let $A'$ denote the rebalancing of $A$ in the direction of $B$.  By Proposition \ref{bead tree properties}, the map $A \mapsto A'$ defines a function $f:  X_r \rightarrow X_{r'}$.  By rebalancing $A'$ in the direction of $\widetilde{B}$, we recover $A$; thus $A' \mapsto A$ is a well-defined inverse of $f$.  Thus $N_{T,r} = |X_r| = |X_{r'}| = N_{T,r'}$.
\end{proof}

\begin{remark}
In view of Corollary \ref{independence of root}, we shall drop the $r$ from the subscript and simply refer to the quantity $N_T$.  It is not difficult to prove Corollary \ref{independence of root} directly, without reference to bead arrangements.
\end{remark}

We are now ready to state the main result of this section.

\begin{theorem}
\label{generator counting}
\begin{align}
\label{gc1}
|\widehat{\mathcal{E}}/\sim | = |RFBA| = \sum_{T \in \mathcal{PT}_n} N_T \\
\label{gc2}
|\widehat{\mathcal{E}}| = |RCBA| = (n!) \cdot P \cdot \sum_{T \in \mathcal{PT}_n} N_T
\end{align}
\end{theorem}

\begin{proof}
The left-hand equalities were proved in Proposition \ref{arrangement-generator compatibility 2}.  By Proposition \ref{colored to free reduction}, Equation (\ref{gc2}) follows immediately from Equation (\ref{gc1}).  Thus it suffices to prove that $|RFBA| = \sum_{T \in \mathcal{PT}_n} N_T$.

If $T$ has a unique balancing root $r$, then by Corollary \ref{bead arrangement rebalancing} the free bead arrangements of class $(T, r)$ are in bijection with the free reduced bead arrangements of class $T$, hence by Lemma \ref{single tree counting} there are $N_T$ reduced free bead arrangements of class $T$.

If $T$ has two balancing roots $r$ and $r'$, then by Corollary \ref{bead arrangement rebalancing} the free bead arrangements of class $(T, r)$ are in bijection with the free bead arrangements of class $(T, r')$ and also with the reduced free bead arrangements of class $T$.  Therefore there are again $N_T$ reduced free bead arrangements of class $T$.
\end{proof}


\section{The Action of $\Xi$}
\label{mutations}
\subsection{Bead Collisions and Mutations}

So far, our combinatorial model has allowed us to determine the size of $\widehat{\mathcal{E}}$.  Our next task is to describe the action of $\Xi$ (defined in Section \ref{Action}) on $\widehat{\mathcal{E}}$ (see Definition \ref{generator action}).  We have not yet shown $A\dgstab$ admits $(d+1)$-orthogonal maximal extensions; instead, we shall define an action on $CBA$ which descends to $\widehat{\mathcal{E}}$, and show that this is the desired action.

The intuition behind this action is as follows:  given a colored bead arrangement $A$, a set $S\subsetneq [n]$ acts on $A$ by sliding all beads of color $i \notin S$ one unit in the counterclockwise direction.  The resulting tuple of beads need not be a colored bead arrangement, as some of the beads may now overlap.  When this happens, we apply various "mutations" to the moved beads by extending or shrinking them depending on the nature of the collision.

\begin{definition}
\label{naive bead action}
Let $BT$ be the set of all $n$-tuples of beads (with overlaps allowed).  Define an action of $\Xi$ on $BT$ as follows.  Let $\sigma \in \mathfrak{S}_n$, $S \subsetneq [n]$, and $T = (B^1, \cdots, B^n) \in BT$.  Then:
\begin{align*}
\sigma \cdot T &:= (B^{\sigma(1)}, \cdots, B^{\sigma(n)})\\
S\cdot T &:= (B^1(- \delta_{1 \notin S}), \cdots, B^n(- \delta_{n \notin S}) )\\
S^{-1}\cdot T &:= (B^1(\delta_{1 \notin S}), \cdots, B^n(\delta_{n \notin S}) )
\end{align*}
We refer to this action as the \textbf{naive action} on bead tuples.
\end{definition}

Clearly, $CBA \subset BT$ is not stable under the naive action.  However, we are able to classify the ways in which collisions can occur between beads.

\begin{definition}
\label{bead collisions}
Let $B^1$ and $B^2$ be beads.

We say that $B^1$ \textbf{has a left collision of Type I} with $B^2$ if $B^1 \cap B^2$ is precisely the left ridge of $B^1$ and the right ridge of $B^2$.

We say that $B^1$ \textbf{has a left collision of Type II} with $B^2$ if the left ridge of $B^1$ coincides with the left ridge of $B^2$, and $B^1(1)$ lies in the well of $B^2$.

We say that $B^1$ \textbf{has a left collision of Type III} with $B^2$ if the right ridge of $B^1$ coincides with the right ridge of $B^2$, and $B^2$ lies in the well of $B^1(1)$.

We define the mirror notion of right collisions by reversing all instances of "left" and "right" in the above definitions, and replacing all positive shifts with negative shifts.  We shall work almost exclusively with left collisions throughout this paper.  When we do not specify left or right, we shall always mean a left collision.

We shall say $B^1$ is Type I, II, or III \textbf{left adjacent} to $B^2$ if $B(-1)$ has a Type I, II, or III left collision with $B_2$.  We define right adjacency analogously.
\end{definition}

\begin{remark}
Let $A = (B^1, \cdots B^n)$ be a colored bead arrangement and $S \subsetneq [n]$.  It is clear that for any $B^i$ and $B^j$ in $A$, $B^i(-1)$ and $B^j(-1)$ do not overlap; thus two beads in $S\cdot A$ can overlap only if one is moved by $S$ and the other remains stationary.  It is easy to verify that a moved bead $B^i(-1)$ and a stationary bead $B^j$ in $S \cdot A$ overlap if and only if $B^i(-1)$ has a left collision with $B^j$.  Similarly, a moved bead $B^i(1)$ and a stationary bead $B^j$ in $S^{-1} \cdot A$ overlap if and only if $B^i(1)$ has a right collision with $B^j$.
\end{remark}

For $\Xi$ to define an action on $CBA$, we must develop a means of correcting collisions.  For each type of collision, we introduce a corresponding mutation that resolves the collision.

\begin{definition}
\label{bead mutations}
Let $B^1 = B_{l_1}(i_1)$, $B^2 = B_{l_2}(i_2)$ be beads.

If $B^1$ has a Type I left collision with $B^2$, define the \textbf{Type I left mutation of} $B^1$ to be the bead $M_I(B^1) = B_{l_1 + l_2}(i_1)$.

If $B^1$ has a Type II left collision with $B^2$, define the \textbf{Type II left mutation of} $B^1$ to be the bead $M_{II}(B^1) = B_{l_2 - l_1}(i_2 -1)$.

If $B^1$ has a Type III left collision with $B^2$, define the \textbf{Type III left mutation of} $B^1$ to be the bead $M_{III}(B^1) = B_{l_1 - l_2}(i_1 - l_2(d+2))$.

Right mutations are defined analogously.  As with collisions, we shall simply write ``mutations'' when referring to left mutations.

If $A = (B^1, \cdots, B^n) \in BT$, and for some $i$ there is a unique $j$ such that $B^i$ has a Type $r$ collision with $B^j$, define $M_r^i(A)$ to be the tuple obtained from $A$ by replacing $B^i$ with $M_r(B^i)$.
\end{definition}

Each of the three types of mutation corresponds to an intuitive physical transformation of the bead.

In a mutation of type I, we extend the length of $B^1$, keeping the right endpoint fixed, until $B^2$ lies in its well.  $B^2$ will be Type II left adjacent to $M_I(B^1)$.  This process is illustrated in Figure \ref{beads5} below.

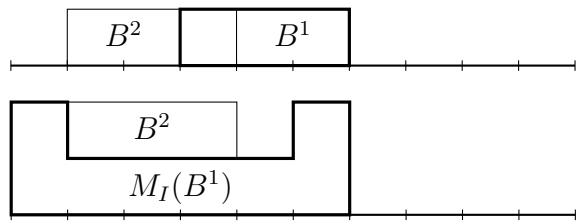
\begin{figure}[H]
\centering
\captionsetup{justification = centering}
\begin{tikzpicture}[x = 0.75 cm, y = 0.75 cm]
\draw[thick] (0,0) -- (10, 0) ;
\foreach \x in {0, 1, 2, 3, 4, 5, 6, 7, 8, 9, 10}
	\draw (\x, 2pt) -- (\x, -2pt);

\draw (1, 0) rectangle (4, 1);
\draw (2, 0.15 ) node[anchor = south] {$B^2$};
\draw[very thick] (3, 0) rectangle (6, 1);
\draw (5, 0.15 ) node[anchor = south] {$B^1$};
\end{tikzpicture}

\vspace{10 pt}

\begin{tikzpicture}[x = 0.75 cm, y = 0.75 cm]
\draw[thick] (0,0) -- (10, 0) ;
\foreach \x in {0, 1, 2, 3, 4, 5, 6, 7, 8, 9, 10}
	\draw (\x, 2pt) -- (\x, -2pt);

\draw (1, 1) rectangle (4, 2);
\draw (2.5, 1.15 ) node[anchor = south] {$B^2$};
\draw[very thick] (0, 0) --  (6, 0) -- (6, 2) -- (5, 2) -- (5, 1) -- (1, 1) -- (1, 2) -- (0, 2) -- cycle;
\draw (3, 0 ) node[anchor = south] {$M_I(B^1)$};
\end{tikzpicture}

\caption{Top:  A Type I left collision of $B^1$ with $B^2$. \newline
Bottom:  A Type I mutation of $B^1$.\newline
$n = 3, d = 1$}
\label{beads5}
\end{figure}

In a mutation of type II, we "reflect" $B^1$ inside the well of $B^2$.  $B^2$ will be Type III left adjacent to $M_{II}(B^1)$, and the left ridge of $M_{II}(B^1)$ will coincide with the right ridge of $B^1$.  $M_{II}(B^1)$ could be described as the ``partner of $B^1$, relative to $B^2$''.  This is illustrated in Figure \ref{beads6}.

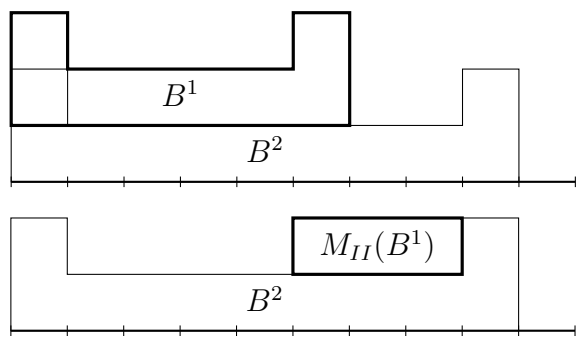
\begin{figure}[H]
\centering
\captionsetup{justification = centering}
\begin{tikzpicture}[x = 0.75 cm, y = 0.75 cm]
\draw[thick] (0,0) -- (10, 0) ;
\foreach \x in {0, 1, 2, 3, 4, 5, 6, 7, 8, 9, 10}
	\draw (\x, 2pt) -- (\x, -2pt);

\draw[very thick] (0, 1) -- (6, 1) -- (6, 3) -- (5, 3) -- (5, 2) -- (1, 2) -- (1, 3) -- (0, 3) -- cycle;
\draw (3, 1.15 ) node[anchor = south] {$B^1$};
\draw (0, 0) --  (9, 0) -- (9, 2) -- (8, 2) -- (8, 1) -- (1, 1) -- (1, 2) -- (0, 2) -- cycle;
\draw (4.5, 0.15 ) node[anchor = south] {$B^2$};
\end{tikzpicture}

\vspace{10 pt}

\begin{tikzpicture}[x = 0.75 cm, y = 0.75 cm]
\draw[thick] (0,0) -- (10, 0) ;
\foreach \x in {0, 1, 2, 3, 4, 5, 6, 7, 8, 9, 10}
	\draw (\x, 2pt) -- (\x, -2pt);

\draw[very thick] (5, 1) rectangle (8, 2);
\draw (6.5, 1 ) node[anchor = south] {$M_{II}(B^1)$};
\draw (0, 0) --  (9, 0) -- (9, 2) -- (8, 2) -- (8, 1) -- (1, 1) -- (1, 2) -- (0, 2) -- cycle;
\draw (4.5, 0.15 ) node[anchor = south] {$B^2$};
\end{tikzpicture}

\caption{Top:  A Type II left collision of $B^1$ with $B^2$. \newline
Bottom:  A Type II mutation of $B^1$.\newline
$n = 3, d = 1$}
\label{beads6}
\end{figure}

In a mutation of type III, we shorten $B^1$, keeping the left endpoint fixed, until $B^2$ no longer overlaps with it.  $B^2$ will be Type I left adjacent to $M_{III}(B^1)$.  This is illustrated in Figure \ref{beads7}.

\begin{figure}[H]
\centering
\captionsetup{justification = centering}
\begin{tikzpicture}[x = 0.75 cm, y = 0.75 cm]
\draw[thick] (0,0) -- (10, 0);
\foreach \x in {0, 1, 2, 3, 4, 5, 6, 7, 8, 9, 10}
	\draw (\x, 2pt) -- (\x, -2pt);

\draw (3, 1) -- (9, 1) -- (9, 3) -- (8, 3) -- (8, 2) -- (4, 2) -- (4, 3) -- (3, 3) -- cycle;
\draw (6, 1.15 ) node[anchor = south] {$B^2$};
\draw[very thick] (0, 0) --  (9, 0) -- (9, 2) -- (8, 2) -- (8, 1) -- (1, 1) -- (1, 2) -- (0, 2) -- cycle;
\draw (4.5, 0.15 ) node[anchor = south] {$B^1$};
\end{tikzpicture}

\vspace{10 pt}

\begin{tikzpicture}[x = 0.75 cm, y = 0.75 cm]
\draw[thick] (0,0) -- (10, 0) ;
\foreach \x in {0, 1, 2, 3, 4, 5, 6, 7, 8, 9, 10}
	\draw (\x, 2pt) -- (\x, -2pt);

\draw (3, 0) -- (9, 0) -- (9, 2) -- (8, 2) -- (8, 1) -- (4, 1) -- (4, 2) -- (3, 2) -- cycle;
\draw (6, 0.15 ) node[anchor = south] {$B^2$};
\draw[very thick] (0,0) rectangle (3, 1);
\draw (1.5, 0 ) node[anchor = south] {$M_{III}(B^1)$};
\end{tikzpicture}

\caption{Top:  A Type III left collision of $B^1$ with $B^2$. \newline
Bottom:  A Type III mutation of $B^1$.\newline
$n = 3, d = 1$}
\label{beads7}
\end{figure}
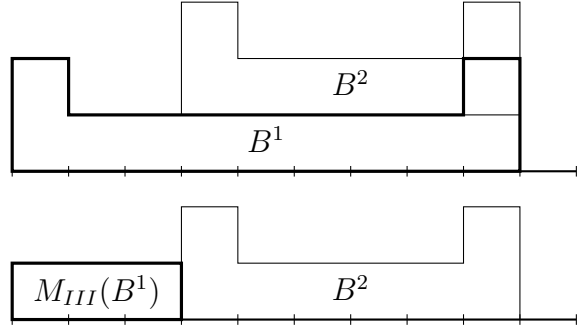

We are now ready to define an action of $\Xi$ on $CBA$.

\begin{definition}
\label{bead action}
Let $\sigma \in \mathfrak{S}_n$, $S \subsetneq \{1, \cdots n\}$, and $A = (B^1, \cdots, B^n) \in CBA$.  Define an action of $\Xi$ on $CBA$ according to the following procedure.  We denote this action by $\circ$ to distinguish it from the naive action $\cdot$ defined in \ref{naive bead action}.\\
1)  Let $S \circ A := S \cdot A$ if $S\cdot A \in CBA$.\\
2)  If $S \cdot A \notin CBA$, apply Type I left mutations to $S\cdot A$ until there are no Type I left collisions between beads in $S\cdot A$.  The mutations may be applied in any order.  Call the resulting tuple $M_I(S\cdot A)$\\
3)  Apply Type III left mutations to $M_I(S\cdot A)$ until there are no Type III left collisions between beads in $M_I(S\cdot A)$.  The mutations may be applied in any order.  Call the resulting tuple $M_{III}M_I(S\cdot A)$\\
4)  Apply Type II left mutations to $M_{III}M_I(S\cdot A)$ until there are no Type II left collisions between beads in $M_{III}M_I(S\cdot A)$.  The mutations may be applied in any order.  Call the resulting tuple $M_{II}M_{III}M_I(S\cdot A)$.\\
5)  Apply Type III left mutations to $M_{II}M_{III}M_I(S\cdot A)$ until there are no Type III left collisions between beads in $M_{II}M_{III}M_I(S\cdot A)$.  The mutations may be applied in any order.  Define the resulting tuple to be $S\circ A$.\\
6)  Define $S^{-1} \circ A$ in analogy with 1-5) above, but replacing the word "left" with "right".\\
7)  Let $\sigma \circ A := \sigma \cdot A$.
\end{definition}

This action is illustrated in Figure \ref{beads8} below.  If the red bead in the upper figure is moved to the left, it undergoes mutations of Type I, III, II, and III, which produces the lower figure.

\begin{remark}
The procedure for computing the action is chosen to simplify the proof that $S\circ A$ is a well-defined colored bead arrangement.  In fact, one can apply mutations to resolve collisions in any order.  The final colored bead arrangement remains the same, as does the total number of mutations.  However, we do not prove this.
\end{remark}

One could define an alternative procedure, in which the Type II mutations occur in Step 3, followed by Type I mutations in Step 4, with the rest of the instructions unchanged.  It is easy to check that applying $M_{II}$ to a bead with a Type III collision produces a bead with a Type I collision, and that $M_IM_{II} = M_{II}M_{III}$.  Thus the two procedures diverge at Step 3 but reconverge at Step 4.  The careful reader may have noticed that the action of $S^{-1}$ does not actually perform the inverse operations of the original procedure in reverse order; instead, it performs the inverse operations of the alternative procedure in reverse order.  Since the two procedures are equivalent, we have that $S$ and $S^{-1}$ act inversely.

\begin{figure}[H]
\centering
\begin{tikzpicture}[x = 0.5 cm, y = 0.5 cm]
\draw[thick] (0,0) -- (16, 0) ;
\foreach \x in {0, 1, 2, 3, 4, 5, 6, 7, 8, 9, 10, 11, 12, 13, 14, 15, 16}
	\draw (\x, 2pt) -- (\x, -2pt);

\draw (1, 0) -- (16, 0) -- (16, 2) -- (15, 2) -- (15,1) -- (2,1) -- (2,2) --(1,2) -- cycle;
\draw (2, 1) rectangle (5, 2);
\draw[red, thick] (5, 1) -- (11, 1) -- (11, 3) -- (10, 3) -- (10, 2) -- (6, 2) -- (6, 3) -- (5, 3) -- cycle;
\draw (7, 2) rectangle (10, 3);
\draw (12, 1) rectangle (15, 2);
\end{tikzpicture}

\vspace{10 pt}

\begin{tikzpicture}[x = 0.5 cm, y = 0.5 cm]
\draw[thick] (0,0) -- (16, 0) ;
\foreach \x in {0, 1, 2, 3, 4, 5, 6, 7, 8, 9, 10, 11, 12, 13, 14, 15, 16}
	\draw (\x, 2pt) -- (\x, -2pt);

\draw (1, 0) -- (16, 0) -- (16, 2) -- (15, 2) -- (15,1) -- (2,1) -- (2,2) --(1,2) -- cycle;
\draw (2, 1) rectangle (5, 2);
\draw[red, thick] (6, 1) -- (12, 1) -- (12, 3) -- (11, 3) -- (11, 2) -- (7, 2) -- (7, 3) -- (6, 3) -- cycle;
\draw (7, 2) rectangle (10, 3);
\draw (12, 1) rectangle (15, 2);
\end{tikzpicture}

\caption{Shift the red bead (top figure) one to the left to produce the bottom figure.  $n = 5$, $d=1$}
\label{beads8}
\end{figure}
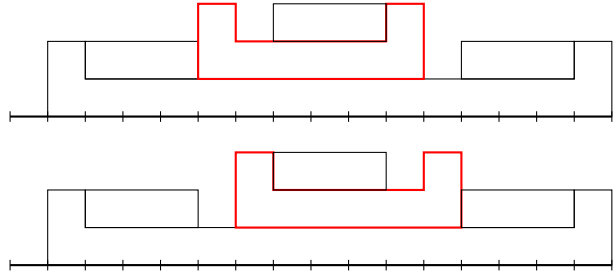

\begin{proposition}
Let all notation be as in Definition \ref{bead action}.   The algorithm defining $S\circ A$ terminates and $S\circ A \in CBA$.  The action of $\Xi$ on $CBA$ is well-defined.
\end{proposition}

If $S\circ A$ is well-defined, it follows by symmetry that $S^{-1} \circ A$ is well-defined.  It is easy to check that the actions of $S$ and $S^{-1}$ are mutually inverse, hence the action of $Free(\mathcal{P}'(n))$ on $CBA$ is well-defined.  It is also clear that the actions of $Free(\mathcal{P}'(n))$ and $\mathfrak{S_n}$ induce an action of $\Xi$ on $CBA$.  Thus it is enough to show that $S\circ A$ is well-defined.  We prove this with a sequence of lemmas below.

We shall refer to the $i$th entry of any tuple obtained from $S\cdot A$ via a sequence of mutations as a \textbf{moved bead} if $i \notin S$ and a \textbf{stationary bead} if $i \in S$.  It is clear that the only overlaps between beads in $S\cdot A$ are left collisions of moved beads with stationary beads.  Furthermore, a moved bead can have collisions with at most two stationary beads:  a Type I or II collision along its left ridge, and a Type III collision along its right ridge.

\begin{lemma}
\label{step 1 validity}
$M_I(S\cdot A)$ exists; that is, there is a unique tuple in $BT$ which has no Type I left collisions and is obtained from $S\cdot A$ by a finite sequence of Type I left mutations.  The only overlaps between beads in $M_I(S\cdot A)$ are Type II or II left collisions of moved beads with stationary beads.
\end{lemma}

\begin{proof}
Write $A = (B^1, \cdots, B^n)$.  Let $B^i(-1)$ be a moved bead in $S\cdot A$ which has a Type I collision with a stationary bead $B^j$.  Apply $M^i_I$ to $S\cdot A$.  The right ridge of $M_I(B^i(-1))$ has not moved and thus causes no new overlaps.  The left ridge of $M_I(B^i(-1))$ is one unit to the left of the left ridge of $B^j$, so that $B^j$ lies in the well of $M_I(B^i(-1))$.

If $M_I(B^i(-1))$ overlaps with some bead $B\neq B^i(-1)$ in $S\cdot A$, but not in a Type I, II, or III collision, then one of the ridges of $B$ intersects $M_I(B^i(-1))$ somewhere other than its left ridge.  The ridge of $B$ cannot overlap with the left ridge of $B^i(-1)$, since $B^i(-1)$ already has a Type I left collision with $B^j$; the ridge of $B$ cannot be in any other position, since this would result in a forbidden overlap with either $B^j$ or $B^i(-1)$ in $S\cdot A$.  Thus $M_I(B^i(-1))$ overlaps with another bead in $M_I^i(S\cdot A)$ if and only if it has a Type I, II, or III collision.

Since applying $M_I^i$ does not move the right ridge of $B^i(-1)$, Type III collisions are not affected by $M_I$.  There are three possibilities for the left ridge:

The first possibility is that the left ridge of $M_I(B^i(-1))$ does not intersect with any other bead.  In this case, $M_I(B^i(-1))$ no longer has a Type I collision with any other bead.

The second possibility is that the left ridge of $M_I(B^i(-1))$ overlaps with the left ridge of a bead $B^k$.  Note that $B^k$ must be a stationary bead, since its left ridge is adjacent to that of the stationary bead $B^j$.  Furthermore, both $B^i$ and $B^j$ must lie in the well of $B^k$.  In this case, $M_I(B^i(-1))$ has a Type II collision with $B^k$, but no longer has any Type I collisions.

The third possibility is that the left ridge of $M_I(B^i(-1))$ overlaps with the right ridge of a bead $B^k$.  Again $B^k$ must be a stationary bead.  Then $M_I(B^i(-1))$ has a Type I collision with $B^k$.  In this case, we apply another Type I mutation to $M_I(B^i(-1))$, repeating the process until we are in either of the first two situations.  Each time we apply a Type I mutation, a new stationary bead is added to the well of $B^i(-1)$; since there are only finitely many such beads, this process must terminate after finitely many steps.

We have shown that for each $i$, there is some $k_i \ge 0$ such that $(M^i_{I})^{k_i}(S\cdot A)$ has no Type $I$ collisions with any other entry.  Furthermore, the only overlaps between beads in $(M^i_{I})^{k_i}(S\cdot A)$ are left collisions of moved beads with stationary beads, and the number of beads with a Type I collision has decreased by one.  Note that the above argument applies verbatim if $S\cdot A$ is replaced with $(M^i_{I})^{k_i}(S\cdot A)$.  Applying the argument at the index of each bead with a Type I collision, we obtain the desired tuple $M_I(S\cdot A)$.

Uniqueness of $M_I(S\cdot A)$ follows immediately from the fact that $M^i_I$ and $M^j_I$ commute for all $i$ and $j$.
\end{proof}

\begin{lemma}
\label{step 2 validity}
$M_{III}M_I(S\cdot A)$ exists; that is, there is a unique tuple in $BT$ which has no Type I or III left collisions and is obtained from $M_I(S\cdot A)$ by a finite sequence of Type III left mutations.  The only overlaps between beads in $M_{III}M_I(S\cdot A)$ are Type II left collisions of moved beads with stationary beads.  
\end{lemma}

\begin{proof}
We use the same notation as Lemma \ref{step 1 validity}.

Let $B^{i'}  = (M_I)^{k_i}(B^i(-1))$ be a moved bead in $M_I(S\cdot A)$ which has a Type III collision with a stationary bead $B^j$.  Apply $M^i_{III}$ to $M_I(S \cdot A)$.  The left ridge of $M_{III}(B^{i'})$ is unchanged and causes no new overlaps; in particular, $M_{III}(B^{i'})$ has no Type I collisions.  The right ridge of $M_{III}(B^{i'})$ is one unit to the right of the left ridge of $B^j$, so the two beads no longer intersect.

Note that it is not possible for the right ridge of $M_{III}(B^{i'})$ to overlap with the left ridge of a bead $B$ in $M_I(S\cdot A)$.  If this were the case, then $B^k$ would overlap with either $B^j$ or $B^{i'}$, and neither overlap would be a collision.  This contradicts our our construction of $M_I(S\cdot A)$.  More generally, it is not possible for $M_{III}(B^{i'})$ to overlap with any bead $B$ in $M_I(S\cdot A)$ except in a left collision. Thus, we need only consider the two possibilities for the right ridge of $M_{III}(B^{i'})$:

The first possibility is that the right ridge of $M_{III}(B^{i'})$ does not overlap with the right ridge of any other bead.  In this case, $M_{III}(B^{i'})$ no longer has a Type III collision.

The second possibility is that the right ridge of $M_{III}(B^{i'})$ overlaps with the right ridge of a bead $B^k$.  Since the right ridge of $B^k$ is adjacent to the stationary bead $B^j$, $B^k$ must be a stationary bead, and $M_{III}(B^{i'})$ has a Type III collision with $B^k$.  Since each Type III mutation removes a bead from the well of $B^{i'}$, there is some $l_i$ such that $(M_{III})^{l_i}(B^{i'})$ has no Type III collisions.

We have shown that for each $i$, there exists $l_i$ such that $(M_{III}^i)^{l_i}M_I(S\cdot A)$ has no Type I or III collisions with any other entry of $M_I(S\cdot A)$.  The only overlap between beads in $(M_{III}^i)^{l_i}M_{I}(S\cdot A)$ are Type II collisions of moved beads with stationary beads, and the number of beads with a Type III collision has decreased by one.  Once again, we can apply the same argument to $M_{III}^{l_i}M_{I}(S\cdot A)$ at each remaining bead with a Type III collision.  This produces the desired tuple $M_{III}M_{I}(S\cdot A)$.

Uniqueness of $M_{III}M_{I}(S\cdot A)$ follows from the fact that $M^i_{III}$ and $M^j_{III}$ commute for all $i$ and $j$.
\end{proof}

\begin{lemma}
\label{step 3 validity}
$M_{II}M_{III}M_I(S\cdot A)$ exists; that is, there is a unique tuple in $BT$ which has no Type I or II left collisions and is obtained from $M_{III}M_I(S\cdot A)$ by a finite sequence of Type II left mutations.  The only overlaps between beads in $M_{II}M_{III}M_I(S\cdot A)$ are Type III left collisions of moved beads with stationary beads.
\end{lemma}

\begin{proof}
We use the notation of Lemma \ref{step 2 validity}.  Write $B^{i''} = M_{III}^{l_i}(B^{i'})$ for each moved bead in $M_{III}M_{I}(S\cdot A)$.

If $B^{i''}$ has a Type II collision with the stationary bead $B^j$, then $M_{II}(B^{i''})$ lies in the well of $B^j$.  The left ridge of $M_{II}(B^{i''})$ coincides with the right ridge of $B^{i''}$, and the right ridge of $M_{II}(B^{i''})$ is adjacent to the right ridge of $B^j$.  Since $B^{i''}$ did not overlap with any bead except $B^j$, the only possible overlap for $M_{II}(B^{i''})$ is a Type III collision with a bead $B^k$ whose right edge lies at the rightmost point of the well of $B^j$.  Since the right edge of $B^k$ is adjacent to the right ridge of the stationary bead $B^j$, $B^k$ must be a stationary bead.

Note that $B^{i''}$ is the only bead in $M_{III}M_{I}(C\cdot A)$ which has a Type II collision with $B^j$; if another moved bead $B^{r''}$ had a Type II collision with $B^j$, then $B^{i''}$ and $B^{r''}$ would overlap, a contradiction.  In particular, for any $B^{r''}$ which has a Type II collision, $M_{II}(B^{r''})$ and $M_{II}(B^{i''})$ do not overlap.

For each $i$ such that $B^{i''}$ has a Type II collision, apply $M_{II}^{i}$ to $M_{III}M_I(S\cdot A)$; call the resulting tuple $M_{II}M_{III}M_I(S\cdot A)$.  We have shown that the only possible overlaps between beads in $M_{II}M_{III}M_{I}(S\cdot A)$ are Type III overlaps between moved beads which have undergone a Type II mutation and stationary beads.

Uniqueness of $M_{II}M_{III}M_I(S\cdot A)$ follows from the fact that $M_{II}^i$ and $M_{II}^j$ commute for all $i$ and $j$.
\end{proof}

\begin{lemma}
$S\circ A$ exists; that is, there is a unique tuple in $CBA$ which is obtained from $M_{II}M_{III}M_I(S\cdot A)$ by a finite sequence of Type III left mutations.
\end{lemma}

\begin{proof}
We use the notation of Lemma \ref{step 3 validity}.  Let $B^{i'''} = M_{II}(B^{i''})$ be a moved bead in $M_{II}M_{III}M_I(S\cdot A)$ which has a Type III collision with some stationary bead $B^k$.  By the previous lemma, $B^{i'''}$ lies in the well of a stationary bead $B^j$.  The left ridge of $B^{i'''}$ does not overlap with any other bead, and the right ridge of $B^{i'''}$ is adjacent to the right ridge of $B^j$ and coincides with the right ridge of $B^k$.  It is clear that $M_{III}(B^{i'''})$ cannot overlap with another bead in $M_{II}M_{III}M_I(S\cdot A)$, except possibly in a Type III collision with a stationary bead $B^l$ which is right adjacent to $B^k$.  Each application of $M_{III}$ removes a bead from the well of $B^{i'''}$, hence there is some $r_i$ such that $(M_{III})^{r_i}(B_{i'''})$ does not overlap with any bead.

Applying $(M^i_{III})^{r_i}$ to $M_{II}M_{III}M_I(S\cdot A)$ at each bead with a Type III collision, we obtain the colored bead arrangement $S\circ A$, which is unique since Type III mutations commute.
\end{proof}


\subsection{Compatibility of Actions}

In this section, we show that $A\dgstab$ admits $(d+1)$-orthogonal maximal extensions, hence $\Xi$ acts on $\widehat{\mathcal{E}}$ by perverse tilts.  (See Sections \ref{Maximal Extensions} and \ref{Action} for definitions and terminology.)  We shall also see that the two actions are compatible via $\Phi$.  We shall need the following three lemmas.

\begin{lemma}
\label{action compatibility 1}
Let $B^1$, $B^2$ be non-overlapping beads.  Then $B^1(-1)$ has a Type $\alpha$ left collision with $B^2$ for some $\alpha \in \{\text{I, II, III}\}$ if and only if $dim\Hom_{A\dgstab}(\Phi(B^1), \Phi(B^2(1))) = 1$.  In this case, any nonzero morphism fits into a triangle
\begin{equation}
\label{mutation triangle}
\Phi(B^2) \rightarrow \Phi(M_{\alpha}(B^1(-1)))(1) \rightarrow \Phi(B^1) \rightarrow \Phi(B^2(1))
\end{equation}

Dually, $B^1(1)$ has a Type $\alpha$ right collision with the bead $B^2$ if and only if $dim_{A\dgstab}\Hom(\Phi(B^2(-1)), \Phi(B^1))=1$.  In this case, any nonzero morphism fits into a triangle
\begin{equation}
\label{mutation triangle right}
\Phi(B^2(-1)) \rightarrow \Phi(B^1) \rightarrow \Phi(M_{\alpha}(B^1(1)))(-1) \rightarrow \Phi(B^2)
\end{equation}
\end{lemma}

\begin{proof}
Let $B^1 = B_{l_1}(i_1), B^2 = B_{l_2}(i_2)$.  If $B^1(-1)$ has a Type I collision with $B^2$, we must have that $i_2 = i_1 - l_1(d+2)$ and $l_1+l_2 \le n$.  By Proposition 6.9 of \cite{brightbill2018differential}, $M^1_{l_2}(-l_1(d+2)) \cong M^{l_1+1}_{l_1+l_2}$.  Thus,
\begin{align*}
\Hom_{A\dgstab}(\Phi(B^{1}), \Phi(B^2(1))) &\cong \Hom_{A\dgstab}(M^{1}_{l_1}, M^1_{l_2}(1-l_1(d+2)))\\
& \cong \Hom_{A\grstab}(M^{1}_{l_1}, \Omega^{-1}M^{l_1+1}_{l_1+l_2})\\
& \cong \Ext^1_{A\grmod}(M^{1}_{l_1}, M^{l_1+1}_{l_1+l_2})
\end{align*}
The last space has dimension one.  Any nonzero generator yields a triangle in $D^b(A\grmod)$, which descends to the following triangle in $A\dgstab$:
\begin{equation*}
M^{1+l_1}_{l_1+l_2} \rightarrow M^1_{l_1+l_2} \rightarrow M^1_{l_1} \rightarrow M^{l_1+1}_{l_1+l_2}(1)
\end{equation*}

Applying $(i_1)$ to the triangle, we have $M^1_{l_1+l_2}(i_1) = \Phi(M_I(B^1(-1)))(1)$ and $M^{l_1+1}_{l_1+l_2}(i_1) \cong M^{1}_{l_2}(i_1 - l_1(d+2)) = \Phi(B^2)$.  We have obtained the desired triangle in $A\dgstab$.

If $B^1(-1)$ has a Type II or III collision with $B^2$, the proof is analogous.

Conversely, if $B^1(-1)$ has no left collision with $B^2$, then $B^1(-1)$ and $B^2$ do not overlap, hence $\Hom_{A\dgstab}(\Phi(B^1(-1)), \Phi(B^2)) = 0$ by Proposition \ref{bead-module compatibility 2}.

The proof for right mutations is dual.
\end{proof}

\begin{lemma}
\label{action compatibility 2}
Let $A = (B^1, \cdots, B^n)$ be a colored bead arrangement.  Let $S \subsetneq [n]$.  Let $S\circ A = (C^1, \cdots, C^n)$.  Suppose $C^i$ is a moved bead and $C^j = B^j$ is a stationary bead.  Then $\Hom_{A\dgstab}(\Phi(C^i)(1), \Phi(B^j)) = 0$.

Dually, if $S^{-1}\circ A = (C^1, \cdots, C^n)$, with $C^i$ a moved bead and $C^j = B^j$ a stationary bead, then $\Hom_{A\dgstab}(\Phi(B^j), \Phi(C^i)(-1)) = 0$.
\end{lemma}

\begin{proof}
Since $C^i$ and $B^j$ are part of a colored bead arrangement, they do not overlap.  Thus either $C^i(1)$ and $B^j$ do not overlap, in which case we are done by Proposition \ref{bead-module compatibility 2}, or $C^i(1)$ has a Type I, II, or III right collision with $B^j$.

Write $C^i = B_{l_i}(k_i)$, $B^j = B_{l_j}(k_j)$.  If there is a Type I right collision, then $k_i = k_j - l_j(d+2)$, hence $\Phi(C^i(1)) = M^1_{l_i}(1+ k_j - l_j(d+2)) \cong M^{1+l_j}_{l_i+l_j}(1 + k_j)$ by Proposition 6.9 of \cite{brightbill2018differential}.  Note that $l_1+l_2 \le n$.  Then,
\begin{align*}
\Hom_{A\dgstab}(\Phi(C^i)(1), \Phi(B^j)) &\cong \Hom_{A\dgstab}(M^{1+l_j}_{l_i+l_j}(1+k_j), M^{1}_{l_j}(k_j))\\
&\cong \Hom_{A\dgstab}(M^{1+l_j}_{l_i+l_j}, M^{1}_{l_j}(-1))\\
&= \Hom_{A\grstab}(M^{1+l_j}_{l_i+l_j}, M^{1+l_j}_{1})\\
&= 0
\end{align*}
The other two cases follow by analogous arguments.

The proof for $S^{-1}$ is dual.
\end{proof}

\begin{lemma}
\label{action compatibility 3}
Let $A = (B^1, \cdots B^n) \in CBA$, $S\subsetneq \{1, \cdots, n\}$.  Let $\mathcal{S} = \langle \{B^j \mid j \in S\} \rangle$.  Write $S\circ A = (C^1, \cdots, C^n)$, and let $i \notin S$.  Then $\Phi(C^i)(1)$ is a maximal extension of $B^i$ by $\mathcal{S}$.

Dually, if $S^{-1} \circ A = (D^1, \cdots, D^n)$, then $\Phi(D^i)(-1)$ is a maximal $\mathcal{S}$-extension by $B^i$.
\end{lemma}

\begin{proof}
To simplify notation, we write $M = \Phi(C^i)(1), N = \Phi(B^i)$.

For any $j \in S$, $\Hom(N, \Phi(B^j)) = 0$, hence $\Hom(N, X) = 0$ for any $X \in \mathcal{S}$.  By Proposition \ref{orthogonal extensions}, $M = N_S$ if and only if we have a morphism $f: M \rightarrow N$ such that $C(f)(-1) \in \mathcal{S}$ and $\Hom(M, X) = \Hom(M(-1), X) = 0$ for all $X \in \mathcal{S}$.

To construct the morphism $f$, note that $C^i$ is obtained by applying a sequence of mutations $M_{\alpha_1}, \cdots, M_{\alpha_r}$ to $B^i(-1)$.  For $1 \le k \le r$, write $N_k = \Phi(M_{\alpha_{k}} \cdots M_{\alpha_{1}}B^i(-1))(1)$.  Define $N_0 = N$ and note that $N_r = M$.  For each $1 \le k \le r$, by Lemma \ref{action compatibility 1} we have a morphism $g_k: N_k \rightarrow N_{k-1}$ which fits into the triangle
\begin{equation*}
\Phi(B^{j_k}) \rightarrow N_k \xrightarrow{g_k} N_{k-1} \rightarrow \Phi(B^{j_k})(1)
\end{equation*}
with $j_k \in S$.  Let $f = g_1 \cdots g_r : M \rightarrow N$.

Since $\Phi(B^{j_k}) \in \mathcal{S}$ for each $k$, it follows from the octahedron axiom and induction on $k$ that $C(f)(-1) \in \mathcal{S}$.  Given $j \in S$, we have that $\Hom(M(-1), \Phi(B^j)) = 0$ since the beads $C^i$ and $B^j$ do not overlap.  That $\Hom(M, \Phi(B^j)) = 0$ is precisely the statement of Lemma \ref{action compatibility 2}.  The corresponding statements with $B^j$ replaced by any $X \in \mathcal{S}$ follow immediately.  Thus $M = N_S$.

The proof of the second statement is dual.
\end{proof}

We have established the following theorem:

\begin{theorem}
\label{well-defined}
$A\dgstab$ admits $(d+1)$-orthogonal maximal extensions, hence $\Xi$ acts on $\mathcal{E}$ and $\widehat{\mathcal{E}}$, as in Definition \ref{generator action}.  Furthermore, $\Phi: CBA \rightarrow \widehat{\mathcal{E}}$ is a morphism of $\Xi$-sets.
\end{theorem}

\begin{proof}
The first statement follows immediately from Lemma \ref{action compatibility 3}.  The lemma also shows that $\Phi(S\circ A) = S \cdot \Phi(A)$ for any colored bead arrangement $A$ and $S\subsetneq [n]$.  It is clear that any $\sigma \in \mathfrak{S}_n$ commutes with $\Phi$, hence $\Phi$ is a morphism of $\Xi$ sets.
\end{proof}


\subsection{Transitivity}

We now prove that the action of $\Xi$ on $\widehat{\mathcal{E}}$ is transitive.  It will then follow easily that $\mathcal{E} = \widehat{\mathcal{E}}$, hence every orthogonal tuple is a generating tuple.  We shall require two definitions:

\begin{definition}
Let $A$ be a colored bead arrangement, and let $S \subsetneq [n]$.  If $S\circ A = S\cdot A$ (i.e., no mutations occur), we say that $S \cdot A$ and $A$ \textbf{differ by an elementary rigid motion}.  We say two colored bead arrangements \textbf{differ by a rigid motion} if they are connected by a finite sequence of elementary rigid motions.  We say two uncolored or free bead arrangements differ by a rigid motion if they are the images of colored bead arrangements which differ by a rigid motion. 
\end{definition}

Note that applying a rigid motion does not affect the class of a bead arrangement.

\begin{definition}
Let $B$ be a  bead in a colored, uncolored, or free bead arrangement.  We say $B$ is \textbf{right-justified} if $B(1)$ has a Type I or II right collision with another bead.  We say a colored, uncolored, or free bead arrangement $A$ is right-justified if there exists a bead $B$ in $A$ such that all beads $B'\neq B$ in $A$ are right-justified.
\end{definition}

The parameter $d$ determines the amount of empty space in the well of each bead in the arrangement, as well as on the ring.  Thus if $d = 0$, all bead arrangements are right-justified.  If $d>0$, then at most $n-1$ beads can be simultaneously right-justified, since there will always be a bead on the wire which is not right-justified.  Thus in any right-justified bead arrangement, the unique bead which is not right-justified must have height one.  Note that $A = (B_1, B_1(-(d+2)), \cdots B_1(-n(d+2)))$ is right-justified, and that $\Phi(A) = (S_1, \cdots, S_n)$.

It is intuitively clear that any bead arrangement can be converted into a right-justified form via a rigid motion: hold one bead on the wire fixed, and slide all other beads to the right as far as they will go.  More formally:

\begin{lemma}
\label{rigid motion lemma 1}
Let $A \in CBA$.  Then there exists a right-justified bead arrangement $A'$ which differs from $A$ by a rigid motion.
\end{lemma}

\begin{proof}
Without loss of generality, let $B^1 \in A$ have height $1$.  I claim that there exist colored bead arrangements $\{A_i\}_{i \ge 0}$ such that $B^1 \in A_i$ for all $i$, $A_{i+1}$ differs from $A_{i}$ by a rigid motion, and all beads $B \neq B^1$ in $A^i$ of height at most $i$ are right-justified.

Let $A_{0} = A$.  Given $A_{i} = (B^1, \cdots, B^n)$, if every bead $B^j \neq B^1$ of height $i+1$ is right-justified, take $A_{i+1} = A_{i}$.  Otherwise, if some $B^j \neq B^1$ is not right-justified, let $S = \{1, \cdots, n\} - \{r \mid B^r \subset B^j\}$ and repeatedly apply $S^{-1}$ to $A_{i}$ until $B^j$ becomes right-justified.  Each application of $S^{-1}$ is a rigid motion; no Type I or II right collisions occur because $B^j$ is not right-justified, and no Type III collisions occur because all beads in the well of $B^j$ are moved by $S^{-1}$.  Furthermore, all beads of height $\le i$ (excluding $B^1$) remain right justified, since we have only moved beads lying inside their wells.  Note also that $B^1$ is never moved.

Repeat the process in the above paragraph until all height $i+1$ (except possibly $B^1$) beads are right-justified.  If $i > 0$, there is finite space in the well of each bead, so the process must terminate after finitely many steps; if $i =0$, since $B^1$ is always held fixed and the ring is finite, the process again terminates after finitely many steps.  Once all height $i+1$ beads (except possibly $B^1$) are right-justified, call the resulting bead arrangement $A_{i+1}$.  The existence of the desired family $\{A_i\}$ follows by induction.  Then $A_n$ is right-justified and differs from $A = A_0$ by a rigid motion.
\end{proof}

\begin{definition}
\label{standard form}
We say that a right-justified colored bead arrangement $A = (B^1, \cdots, B^n)$ is in \textbf{standard form} if:\\
1)  After identifying each bead with its right endpoint, $B^1 > B^2 > \cdots > B^n > B^1$ with respect to the cyclic order on $\mathbb{Z}/P\mathbb{Z}$.\\
2)  For all $j \neq 1$, $B^j$ is right-justified.
\end{definition}

Any right-justified colored bead arrangement $A = (B^1, \cdots, B^n)$ can be put in standard form by permuting its entries.  Furthermore, if $A$ is in standard form, then $A$ is uniquely determined by two pieces of data:  the associated tree $(P(A), r_A)$ and the choice of $B^1$:  Once the position of $B^1$ is known, one can reconstruct $A$ as an uncolored bead arrangement by simply placing each bead specified by $P(A)$ as far to the right as possible.

\begin{lemma}
\label{rigid motion lemma 2}
Up to permutation of indices, any two right-justified colored bead arrangements of the same class differ by a rigid motion.
\end{lemma}

\begin{proof}
Let $A, A' \in CBA$ be right-justified colored bead arrangements of the same class.  Let $A = (B^1, \cdots, B^n), A' = (C^1, \cdots, C^n)$.  Permuting the entries of $A$, we may assume without loss of generality that $A$ is in standard form.  Permuting $A'$, we may assume that the isomorphism mapping $P(A)$ to $P(A')$ sends $B^i$ to $C^i$ for each $i$.  Note that $A'$ need not be in standard form; condition 1) of Definition \ref{standard form} will be satisfied, but not necessarily condition 2).

If condition 2) is not satisfied, then some $C^{j_1} \neq C^1$ is the unique bead (necessarily of height 1) which is not right justified.  Let $C^{j_1} > C^{j_2} > \cdots C^{j_k} > C^{j_1}$ be the height $1$ beads of $A'$, ordered cyclically by their right endpoints; note that $C^1 = C^{j_r}$ for some $r$.  Applying a rigid motion, we may translate $C^{j_1}$ (and all beads lying in its well) until $C^{j_1}$ is right justified, with its right edge adjacent to $C^{j_k}$.  The resulting colored bead arrangement is right justified and $C^{j_2}$ is now the unique bead which is not right-justified.  Repeat this procedure until $C^{j_r} = C^1$ is not right-justified.  Denote the new arrangement $A''$; clearly $A''$ is in standard form.

It is clear that $A''$ differs from $A'$ by a rigid motion; consequently, $A''$ has the same class as $A'$ and $A$.  Since $C^1$ has the same type as $B^1$, we can write $B^1 = C^1(i)$ for some $i$.  Then $A''(i)$ and $A$ have the same class, are both in standard form, and have the same first bead, hence $A''(i) = A$.  Since $A$ and $A''$ differ by a rigid motion, so do $A$ and $A'$.
\end{proof}

\begin{lemma}
\label{rigid motion lemma 3}
Two colored bead arrangements have the same class if and only if they differ by a rigid motion and a permutation of indices.  Two uncolored or free bead arrangements have the same class if and only if they differ by a rigid motion.
\end{lemma}

\begin{proof}
Applying an elementary rigid motion does not change the height of any bead, nor the relative ordering of the beads' right edges.  Thus two bead arrangements which differ by an elementary rigid motion have the same class, hence also for arbitrary rigid motions.

Conversely, let $A, A'$ be two colored bead arrangements of the same class.  By Lemma \ref{rigid motion lemma 1}, after changing $A$ and $A'$ up to a rigid motion, we can assume without loss of generality that both $A$ and $A'$ are right-justified.  By Lemma \ref{rigid motion lemma 2}, $A$ and $A'$ differ by a rigid motion and a permutation, and we are done.

The second statement follows immediately from the first.
\end{proof}

\begin{theorem}
\label{bead transitivity}
The action of $\Xi$ on $CBA$ is transitive.
\end{theorem}

\begin{proof}
Let $A = (B_1, B_1(-(d+2)), \cdots B_1(-n(d+2)))$.  We shall show that the orbit of $A$ is $CBA$.  By Lemma \ref{rigid motion lemma 3}, it suffices to show that the orbit of $A$ contains one representative of every class of colored bead arrangement.

Let $(T, r)$ be a rooted plane tree with $n+1$ vertices.  Let $V_{T,i}$ denote the set of vertices of $T$ of depth $i$.  Let $T_{\le i}$ denote the subtree of $T$ consisting of all vertices of depth $\le i$.  We shall construct a sequence $\{A_i\}_{i\ge 0}$ of bead arrangements with the following properties:\\
1)  $A_{i+1} = \alpha \circ A_i$ for some $\alpha \in \Xi$.\\
2)  For each $i$, there is an isomorphism $\phi_i: (T_{\le i}, r) \xrightarrow{\sim} (\mathcal{P}(A_i)_{\le i}, r_{A_i})$ of rooted plane trees.\\
3)  $\phi_i$ preserves weight; i.e., $W_T(v) = W_{\mathcal{P}(A_i)}(\phi_i(v))$ for each vertex $v \in T_{\le i}$.\\
4)  All beads in $A_i$ of height $i+1$ are of type 1.\\
5)  $A_i$ is right-justified and in standard form.

Let $A_0 = A$.  Suppose we have constructed $A_i  = (B^1, \cdots, B^n)$ for some $i \ge 0$.  Let $v \in T$ be a vertex of height $i$.  (If no such vertex exists, let $A_{i+1} = A_i$.)  Let $B^j = \phi_i(v)$ (if $i>0$).  Let $c_T(v) = \{v_1 > \cdots > v_r\}$ be the children of $v$.  Let $N_0 = 1$, and let $N_s = 1 + \sum_{c=1}^s W_T(v_c)$ for $s \ge 1$.  By 3) and 4), $\phi_i(v) \in A_i$ has weight $N_r$, and there are $N_r -1$ type 1 beads, $B^{j+1} > B^{j+2} > \cdots > B^{j+N_r -1}$ in the well of $B^j$.  (If $i = 0$, then $\phi_i(v) = r_A$ and the beads $B^{j+c}$ lie on the ring.  All superscripts are then taken modulo $n$.)  Since $A_i$ is right-justified in standard form, the beads $B^{j +c}, 1 \le c < N_r$ are adjacent to one another, so that $B^{j+s}(-1)$ has a Type I left collision with $B^{j+s+1}$, and $B^{j+1}$ is adjacent to the right ridge of $B^j$ (if $i>0$).

Let $S = \{j + c \mid 1 \le c < N_r, c\neq N_s \text{ for any } 0 \le s < r \}$.  When we apply $S$ to $A_i$, the moved beads are $B^j$, every bead not in the well of $B^j$ (if $i>0$), and each bead of the form $B^{j+N_s}$, for $0 \le s < r$.  Note that none of the beads outside the well of $B^j$ have collisions, since the only stationary beads are in the well of $B^j$.  By the same reasoning, $B^j$ can only have a Type III left collision, and this does not happen since $B^{j+1}$ is also a moved bead.  For each $0 \le s <r$, the bead $B^{j+N_s}$ undergoes $W_T(v_{s+1}) -1$ Type I left mutations, which place the beads $B^{j+N_s+1}, \cdots B^{j+ N_{s+1} -1}$ into its well; no Type II collisions are possible since $B^j$ is moved, and no Type III collisions are possible since the $B^{j+N_s}$ are of type 1.  Thus in $\mathcal{P}(S\circ A_i)$, the children of $\phi_i(v)$ have weight $W_T(v_s)$, and are arranged in the same order as the children of $v$.  After performing this process at each height $i$ vertex $v$ of $T$ and converting the tuple into right-justified standard form via a rigid motion and a permutation of indices, denote the resulting colored bead arrangement $A_{i+1}$.

By construction, $A_{i+1}$ is obtained from $A_i$ by application of an element of $\Xi$, $A_{i+1}$ is right-justified and in standard form, and all beads of height $i+2$ in $A_{i+1}$ are of type 1.  Note that $\mathcal{P}(A_{i+1})_{\le i} = \mathcal{P}(A_i)_{\le i}$.  It follows from the preceding paragraph that the map $\phi_i : T_{\le i} \rightarrow \mathcal{P}(A_{i+1})_{\le i}$ can be extended to an isomorphism $\phi_{i+1}: T_{\le i+1} \rightarrow \mathcal{P}(A_{i+1})_{\le i+1}$ which preserves the weight of all vertices for which it is defined.  Thus $A_{i+1}$ satisfies properties 1-5) above, hence the sequence $\{A_i\}$ exists.  Then $A_n$ is of class $(T, r)$ and lies in the orbit of $A$.

This process is illustrated in Figure \ref{beads9} below.
\end{proof}

\begin{figure}[H]
\centering
\begin{tikzpicture}[x = 0.5 cm, y = 0.5 cm]
\draw[thick] (0,0) -- (13, 0) ;
\foreach \x in {0, 1, 2, 3, 4, 5, 6, 7, 8, 9, 10, 11, 12, 13}
	\draw (\x, 2pt) -- (\x, -2pt);

\draw[black] (10, 0) rectangle (13, 1);
\draw[red, thick] (7, 0) rectangle (10, 1);
\draw[blue, thick] (4, 0) rectangle (7, 1);
\draw[green, thick] (1, 0) rectangle (4, 1);
\end{tikzpicture}

\vspace{10 pt}

\begin{tikzpicture}[x = 0.5 cm, y = 0.5 cm]
\draw[thick] (0,0) -- (13, 0) ;
\foreach \x in {0, 1, 2, 3, 4, 5, 6, 7, 8, 9, 10, 11, 12, 13}
	\draw (\x, 2pt) -- (\x, -2pt);

\draw[black] (0, 0) -- (12, 0) -- (12, 2) -- (11, 2) -- (11, 1) -- (1, 1) -- (1, 2) -- (0, 2) --cycle;
\draw[red, thick] (8, 1) rectangle (11, 2);
\draw[blue, thick] (5, 1) rectangle (8, 2);
\draw[green, thick] (2, 1) rectangle (5, 2);
\end{tikzpicture}

\vspace{10 pt}

\begin{tikzpicture}[x = 0.5 cm, y = 0.5 cm]
\draw[thick] (0,0) -- (13, 0) ;
\foreach \x in {0, 1, 2, 3, 4, 5, 6, 7, 8, 9, 10, 11, 12, 13}
	\draw (\x, 2pt) -- (\x, -2pt);

\draw[black] (0, 0) -- (12, 0) -- (12, 2) -- (11, 2) -- (11, 1) -- (1, 1) -- (1, 2) -- (0, 2) -- cycle;
\draw[red, thick] (5, 1) -- (11, 1) -- (11, 3) -- (10, 3) -- (10, 2) -- (6, 2) -- (6, 3) -- (5, 3) -- cycle;
\draw[blue, thick] (7, 2) rectangle (10, 3);
\draw[green, thick] (2, 1) rectangle (5, 2);
\end{tikzpicture}

\caption{Arrangments $A_0$, $A_1$, $A_2$ ; $n = 4$, $d=1$}
\label{beads9}
\end{figure}
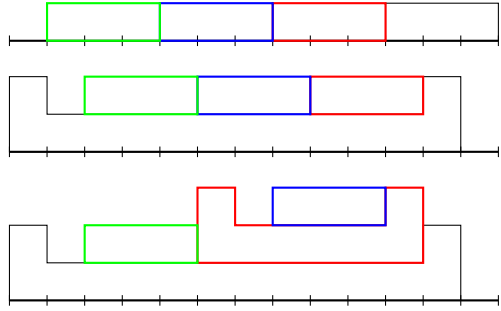

\begin{corollary}
\label{object transitivity}
The action of $\Xi$ on $\widehat{\mathcal{E}}$ is transitive, and $\widehat{\mathcal{E}} = \mathcal{E}$.
\end{corollary}

\begin{proof}
Transitivity follows from the fact that $\Xi$ acts transitively on $CBA$ and $\Phi: CBA \rightarrow \widehat{\mathcal{E}}$ is a surjective morphism of $\Xi$-sets.  Since $\mathcal{E} \subset \widehat{\mathcal{E}}$ is nonempty and stable under $\Xi$, it follows from transitivity that the two $\Xi$-sets are equal.
\end{proof}

\section{Connections to Simple-Minded Systems}
\label{Connections}

In this section, we discuss similarities between our work and that of Coelho Sim\~oes and Pauksztello.  Let $k$ be an algebraically closed field, and let $(\mathcal{T}, \Sigma)$ be a $k$-linear, Hom-finite, Krull-Schmidt triangulated category.

For $\mathcal{S}$ a set of objects in $\mathcal{T}$, let $\textsf{add}(\mathcal{S})$ denote the smallest full subcategory containing $\mathcal{S}$ which is closed under isomorphisms, finite coproducts, and finite direct summands.

\begin{definition}{(\cite{simoes2018simple}, Definition 2.1)}
Let $(X_1, \cdots, X_n)$ be a tuple of objects in $\mathcal{T}$.  We say $(X_1, \cdots, X_n)$ is a $|w|$-\textbf{simple-minded system} if it is $|w|$-orthogonal and $\mathcal{T} = \textsf{add}\langle \{\Sigma^{-m}X_i \mid 1 \le i \le n, 0 \le m < |w|\} \rangle$.
\end{definition}

A priori, being a $|w|$-simple-minded tuple is a slightly weaker condition than being a $|w|$-basis, but \cite{simoes2018simple}, Lemma 2.8, shows that the two definitions are equivalent.

$|w|$-simple-minded systems admit a theory of mutation \cite{simoes2017mutations}, \cite{dugas2015torsion} in which subsets of the system are altered and replaced to form a new system.  The alteration process involves approximation theory:  

\begin{definition}
Let $\mathcal{S}$ be a full subcategory of $\mathcal{T}$.  Let $X \in \mathcal{T}, S\in \mathcal{S}$.\\
1)  We say a morphism $\phi: X \rightarrow S$ is a \textbf{left $\mathcal{S}$-approximation of $X$} if $\Hom(S, S') \rightarrow \Hom(X, S')$ is surjective for any $S' \in \mathcal{S}$.  We say a left approximation $\phi$ is \textbf{minimal} if any endomorphism $g$ of $S$ satisfying $g\phi = \phi$ is an automorphism.\\
2)  Dually, we say a morphism $\phi: S \rightarrow X$ is a \textbf{right $\mathcal{S}$-approximation of $X$} if $\Hom(S', S) \rightarrow \Hom(S', X)$ is surjective for any $S' \in \mathcal{S}$.  We say a right approximation $\phi$ is a \textbf{minimal} if any endomorphism $g$ of $S$ satisfying $\phi g = \phi$ is an automorphism.\\
\end{definition}

Given a $|w|$-simple-minded system $(X_1, \cdots, X_n)$, let $S \subsetneq [n]$ and let $\mathcal{S} = \langle \{X_s \mid s \in S\} \rangle$.  For each $i \notin S$, let $f_i: \Sigma^{-1} X_i \rightarrow S_i$ be a minimal left $\mathcal{S}$-approximation of $\Sigma^{-1}X_i$, and define $X_i' := \Sigma^{-1}C(f_i)$.  For $i \in S$, let $X_i' = X_i$.  The resulting tuple $(X_i')_i$ is $|w|$-simple-minded and is called the \textbf{left mutation of $(X_i)_i$ at $S$}.  One defines right mutations dually, using the cone of a minimal right $\mathcal{S}$-approximation of $\Sigma X$.

We have the following relationship between minimal approximations and minimal extensions:
\begin{proposition}
Let $\mathcal{S}$ be a full subcategory of $\mathcal{T}$.\\
1)  Let $f:  X_\mathcal{S} \rightarrow X$ be the minimal extension of $X$ by $\mathcal{S}$.  Then $\phi:  \Sigma^{-1}X \rightarrow \Sigma^{-1}C(f)$ is a minimal left $\mathcal{S}$-approximation of $\Sigma^{-1}X$.\\
2)  Let $f:  X \rightarrow X^\mathcal{S}$ be the minimal $\mathcal{S}$-extension by $X$.  Then $\phi:  C(f) \rightarrow \Sigma X$ is a minimal right $\mathcal{S}$-approximation of $\Sigma X$.
\end{proposition}

\begin{proof}
By Definition \ref{maximal extension}, $\Sigma^{-1}C(f) \in \mathcal{S}$ and for any $S' \in \mathcal{S}$, the map $\phi^*: \Hom(\Sigma^{-1}C(f), S') \rightarrow \Hom(\Sigma^{-1}X, S')$ is an isomorphism.  It follows immediately that $\phi$ is a left $\mathcal{S}$-approximation of $X$.  To show minimality, let $g$ be an endomorphism of $\Sigma^{-1}C(f)$ such that $g\phi = \phi$.  It follows immediately from injectivity of $\phi^*$ that $g = id$.

The proof of 2) is dual.
\end{proof}

An immediate consequence of the above proposition is that given any $|w|$-simple-minded system $(X_i)_i$ and $S\subsetneq [n]$, the action of $S$ on $(X_i)_i$ is precisely the left mutation at $S$.

\subsection{Beads and Arcs}

Let $k$ be an algebraically closed field, and let $w< 0, n \ge 2$ be integers.  Write $d = |w|-1$ and $P = (n+1)(d+2)-2$.  Let $kA_n$ denote the path algebra of the type $A_n$ Dynkin quiver.  Then the trivial extension algebra of $kA_n$ is a Brauer tree algebra which is derived equivalent to the star with $n$ edges.  In \cite{simoes2015hom}, Coelho Sim\~oes studies $|w|$-Hom configurations (equivalent in this setting to $|w|$-simple-minded systems) in a $w$-Calabi-Yau category $\mathcal{C}_{|w|}(kA_n)$ which is obtained as an orbit category of the bounded derived category of $kA_n$.  

Coelho Sim\~oes constructs a geometric model of $\mathcal{C}_{|w|}(Q)$ in which objects are represented by certain diagonals of a $P$-gon.  More specifically, objects are represented by $(d+2)$-diagonals, which connect vertices separated by $k(d+2) -1$ edges for some $1 \le k \le n$.  A $|w|$-Hom-configuration is represented by $n$ pairwise non-intersecting $(d+2)$-diagonals.

Comparison of the Auslander-Reiten quivers reveals that $\mathcal{C}_{d+1}(kA_n)$ is equivalent to $A_{n,d}\dgstab$, where $A_{n,d}$ is the Brauer tree algebra on the star with $n$ edges and socle graded in degree $-d$.  There is a two-to-one map from beads to $(d+2)$-diagonals:  to the bead $B_l(i) = [[i-l(d+2), i]]$ one associates the diagonal connecting $i$ and $i-l(d+2)+1$, i.e., one draws a line connecting the right endpoints of the ridges of the bead.  It is easy to verify that a bead and its partner map to the same diagonal, and that two beads overlap if and only if the corresponding diagonals cross.  The three types of bead mutation correspond to the Ptolemy Arcs of class II described in \cite{simoes2017mutations}, Figure 3.

\section{Acknowledgements}

I would like to thank Rapha\"el Rouquier for his advice and constant support during the writing of this paper.  I would also like to thank Bon-Soon Lin for sharing his combinatorial wisdom with me, as well as Raquel Coelho Sim\~oes and David Pauksztello for helpful discussions regarding connections to their work.

\bibliography{References}{}
\bibliographystyle{plain}

\end{document}